\documentclass[11pt]{amsart}
\usepackage[all]{xy}
\usepackage{amssymb}

\newtheorem{thm}{Theorem}[section]
\newtheorem{lem}[thm]{Lemma}
\newtheorem{cor}[thm]{Corollary}
\newtheorem{prop}[thm]{Proposition}

\theoremstyle{definition}
\newtheorem{example}[thm]{Example}

\newtheorem{defn}[thm]{Definition}

\numberwithin{equation}{thm}
\begin{document}

\title[Generalized higher cluster categories]{Cluster tilting objects in generalized \\ higher cluster categories}

\author{Lingyan GUO}
\address{Universit\'e Paris Diderot - Paris~7, UFR de Math\'ematiques,
Institut de Math\'ematiques de Jussieu, UMR 7586 du CNRS, Case 7012,
B\^atiment Chevaleret, 75205 Paris Cedex 13, France}
\email{guolingyan@math.jussieu.fr}

\date{\today}

\begin{abstract}
We prove the existence of an $m$-cluster tilting object in a
generalized $m$-cluster category which is $(m+1)$-Calabi-Yau and
Hom-finite, arising from an $(m+2)$-Calabi-Yau dg algebra with
finite-dimen-\\sional homology in degree $0$. This is a
generalization of the result for the {$m = 1$} case in Amiot's
Ph.~D.~thesis. Our results apply in particular to higher cluster
categories associated to suitable finite-dimensional algebras of
finite global dimension, and higher cluster categories associated to
Ginzburg dg categories coming from suitable graded quivers with
superpotential.
\end{abstract}

 \maketitle

\section{Introduction}

In recent years, the categorification of cluster algebras has
attracted a lot of attention. Notice that there are two quite
different notions of categorification: additive categorification,
studied in many articles, and monoidal categorification as
introduced in \cite{HL09}. One important class of categories arising
in additive categorification is that of the cluster categories
associated to finite-dimensional hereditary algebras. These were
introduced in \cite{BMRRT06} (for quivers of type $A$ in
\cite{CCS06}), and investigated in many subsequent articles, e.g.
\cite{BMR06} \cite{BMRT07} \cite{CC06} \cite{CK06} \cite{CK08}
\ldots\ , cf. \cite{R10} for a survey. The cluster category
${\mathcal {C}}_Q$ associated to the path algebra of a finite
acyclic quiver $Q$ is constructed as the orbit category of the
bounded derived category ${\mathcal {D}}$$^b ({\rm mod} kQ)$ under
the action of the autoequivalence $\tau^{-1}\Sigma$, where $\Sigma$
is the suspension functor and $\tau$ the Auslander-Reiten
translation. This category is Hom-finite, triangulated and
2-Calabi-Yau. Analogously, for a positive integer $m$, the
$m$-cluster category ${\mathcal {C}}_Q^{\tiny{(m)}}$ is constructed
as the orbit category of ${\mathcal {D}}$$^b ({\rm mod} kQ)$ under
the action of the autoequivalence $\tau^{-1}\Sigma^m$. This higher
cluster category is Hom-finite, triangulated, and
$(m+1)$-Calabi-Yau. It was first mentioned in \cite{Ke05}, and has
been studied in more detail in several articles \cite{ABST}
\cite{KR07} \cite{KR08} \cite{Th07} \ldots\ . Many results about
cluster categories can be generalized to $m$-cluster categories. In
particular, combinatorial descriptions of higher cluster categories
of type $A_n$ and $D_n$ are studied in \cite{KM1} \cite{KM2}, the
existence of exchange triangles in $m$-cluster categories was shown
in \cite{IY08}, and \cite{ZZ} proved that there are exactly $m+1$
non isomorphic complements to an almost complete tilting object, and
so on.

Claire Amiot \cite{Am08} generalized the construction of the cluster
categories to finite-dimensional algebras $A$ of global dimension
$\leq 2$. In order to show that there is a triangle equivalence
between ${\mathcal {C}}_A$, constructed as a triangulated hull
\cite{Ke05}, and the quotient category per$\Pi_3 (A)/{\mathcal
{D}}^b$$\Pi_3 (A)$, where $\Pi_3 (A)$ is the 3-derived preprojective
algebra \cite{Ke09} of $A$, she first studied the category
${\mathcal {C}}_A = $per$A/{\mathcal {D}}$$^b (A)$ associated to a
dg algebra $A$ with the following four properties: \begin{itemize}
\item[1)] $A$ is homologically smooth; \item[2)] $A$ has vanishing homology in
positive degrees; \item[3)] $A$ has finite-dimensional homology in
each degree and \item[4)] $A$ is $3$-Calabi-Yau as a bimodule.
\end{itemize} She proved that the category $\mathcal {C}$ is
Hom-finite and 2-Calabi-Yau. Moreover, the image of the free dg
module $A$ is a cluster tilting object in $\mathcal {C}$ whose
endomorphism algebra is the zeroth homology of $A$. She applied
these results in particular to the Ginzburg dg algebras
$\Gamma_{(Q,W)}$ associated \cite{Gi06} to Jacobi-finite quivers
with potential $(Q,W)$, and then introduced generalized cluster
categories ${\mathcal {C}}_{(Q,W)}$ = per$\Gamma_{(Q,W)}/{\mathcal
{D}}$$^b \Gamma_{(Q,W)}$, which specialize to the cluster categories
${\mathcal {C}}_Q$ in the case where $Q$ is acyclic and $W$ is the
zero potential.

The motivation of this article is to investigate the existence of
cluster tilting objects in generalized higher cluster categories. We
change the above fourth property of the dg algebra $A$ to:
\begin{itemize} \item[4')] $A$ is $(m+2)$-Calabi-Yau as a bimodule. \end{itemize}
Similarly as in \cite{Am08}, using the inherited $t$-structure on
per$A$ we prove that the quotient category ${\mathcal {C}}_A =
$per$A/{\mathcal {D}}$$^b (A)$ is Hom-finite and $(m+1)$-Calabi-Yau
in section 2. We call it the generalized $m$-cluster category. The
image of the free dg module $A$ is an $m$-cluster tilting object in
$\mathcal {C}$ whose endomorphism algebra is the zeroth homology of
$A$.

We apply these main results in section 3 to higher cluster
categories ${\mathcal {C}}_{\tiny{(Q,W)}}^{(m)}$ associated to
Ginzburg dg categories \cite{Ke09} arising from suitable graded
quivers with superpotential $(Q,W)$. In order for the Ginzburg dg
categories to satisfy the four properties, we assume that their
zeroth homologies are finite-dimensional, that the graded quivers
are concentrated in nonpositive degrees, and that the degrees of the
arrows of $Q$ are greater than or equal to $-m$. This generalized
higher cluster category ${\mathcal {C}}_{\tiny{(Q,W)}}^{(m)}$
specializes to the higher cluster category ${\mathcal {C}}_Q^{(m)}$
when $Q$ is an acyclic ordinary quiver and $W$ is the zero
superpotential.

In the last section, we work with finite-dimensional algebras $A$ of
global dimension $\leq n$. If the functor ${\rm Tor}_n^A (-,DA)$ is
nilpotent, then the $(n-1)$-cluster category ${\mathcal
{C}}_A^{(n-1)}$ defined as in section 4 of $A$ is Hom-finite,
$n$-Calabi-Yau and the image of $A$ is an $(n-1)$-cluster tilting
object in ${\mathcal {C}}_A^{(n-1)}$. This section is a
straightforward generalization of Section~4 in \cite{Am08}, so we
only list the main steps of the proof.

\subsection*{Acknowledgments} The author is supported by
the China Scholarship Council (CSC). This article is part of her
Ph.~D.~thesis under the supervision of Bernhard Keller. She is
grateful to him for his guidance, patience and kindness. She also
deeply thanks Dong Yang for helpful discussions and constant
encouragement.

\vspace{.3cm}
\section{Existence of higher cluster tilting objects}

Let $k$ be a field and $A$ a differential graded (dg) $k$-algebra.
We write per$A$ for the {\em perfect derived category} of $A$,
i.e.~the smallest triangulated subcategory of the derived category
${\mathcal {D}}(A)$ containing $A$ and stable under passage to
direct summands. We denote by ${\mathcal {D}}$$^b (A)$ the bounded
derived category of $A$ whose objects are those of ${\mathcal
{D}}(A)$ with finite-dimensional total homology, and denote by $A^e$
the dg algebra $A^{op} \otimes_k A$. Usually, we write [1] for the
suspension functors $\Sigma$ in triangulated categories. Let $m \geq
1$ be a positive integer. Suppose that $A$ has the following
properties ($\star$):
\begin{itemize}
\item[a)] $A$ is homologically smooth, {\it i.e.}~$A$ belongs to per($A^e$)
when considered as a bimodule over itself;

\item[b)] the $p$-th homology $H^p A$ vanishes for each positive integer $p$\,;

\item[c)] the $0$-th homology $H^0 A$ is finite-dimensional;

\item[d)] $A$ is $(m+2)$-Calabi-Yau as a bimodule, {\it i.e.}~there is an
isomorphism in $\mathcal{D}$$(A^e)$
\end{itemize}

\begin{center}
$R {\rm Hom}_{A^e} (A, A^e) \simeq A[-m-2].$
\end{center}

\smallskip

Let $D$ denote the duality functor ${\rm Hom}_k ( - , k)$.

\begin{lem} [\cite{Ke08}]\label{17}

Suppose that $A$ is homologically smooth. Define

\begin{center}
$\Omega = R {\rm Hom}_{A^e} (A, A^e)$
\end{center}

\noindent and view it as an object in $\mathcal{D}$$(A^e)$. Then for
all objects $L$ of $\mathcal{D}$$(A)$ and $M$ of $\mathcal{D}$$^b
(A)$, we have a canonical isomorphism

\begin{center}
$D {\rm Hom}_{{\mathcal{D}}(A)} (M, L) \simeq {\rm
Hom}_{{\mathcal{D}}(A)} (L \overset{L}{\otimes}_A {\Omega}, M).$
\end{center}

\noindent If we have an isomorphism $\Omega \simeq A [-d]$ in
$\mathcal{D}$$(A^e)$ for some positive integer $d$, then
$\mathcal{D}$$^b (A)$ is $d$-Calabi-Yau, {\it i.e.}~we have

\begin{center}
$D {\rm Hom}_{{\mathcal{D}}(A)} (M, L) \simeq {\rm
Hom}_{{\mathcal{D}}(A)} (L, M [d]).$
\end{center}

\end{lem}

From the proof given in \cite{Ke08} of lemma \ref{17}, one can see
that $\mathcal{D}$$^b (A)$ is Hom-finite and is contained in per$A$.
We denote by $\pi$ the canonical projection functor from per$A$ to
${\mathcal{C}}_A =\, $ {per$A$}/{$\mathcal{D}$$^b (A)$}.

\smallskip
The main generalized result is the following theorem:
\begin{thm}\label{1}
Let $A$ be a dg $k$-algebra with the properties ($\star$).

\begin{itemize}
\item[1)] The category ${\mathcal{C}}_A =\, $ {{\rm
per}$A$}/{$\mathcal{D}$$^b (A)$} is {\rm Hom}-finite and
$(m+1)$-Calabi-Yau.

\item[2)] The object $T = {\pi} A$ is an {\em $m$-cluster
tilting object}, i.e.~we have $${\rm Hom}_{{\mathcal{C}}_A} (T,
T[r]) = 0, \, r = 1,\ldots,m,$$ and for each object $L$ in
${\mathcal {C}}_A$, if ${\rm Hom}_{{\mathcal{C}}_A} (T, L [r])$
vanishes for all $r = 1,\ldots,m ,$ then $L$ belongs to $add(T)$ the
full subcategory of ${\mathcal {C}}_A$ consisting of direct factors
of direct sums of copies of $\pi A$,

\item[3)] The endomorphism algebra  of $T$ is isomorphic to $H^0 A$.
\end{itemize}
\end{thm}

We call ${\mathcal {C}}_A$ the {\em generalized $m$-cluster
category} associated to $A$.

We simply denote ${\mathcal{D}} (A)$ by $\mathcal{D}$, and denote
${\mathcal {C}}_A$ by ${\mathcal {C}}$.

Let ${\mathcal{D}}^{{\leq} {0}}$ (resp. ${\mathcal{D}}^{{\geq}
{1}}$) be the full subcategory of $\mathcal{D}$ whose objects are
the dg modules $X$ such that $H^p X$ vanishes for all $p > 0$ (resp.
$p {\leq} 0$). Similar as in {\cite{Am08}}, the proof of Theorem
$\ref{1}$ also depends on the existence of a canonical $t$-structure
(${\mathcal{D}}^{{\leq} {0}} , {\mathcal{D}}^{{\geq} {1}}$) in
per$A$. For a complex of $k$-modules $X$, we denote by
${\tau}_{{\leq} 0} X$ the subcomplex with $({\tau}_{{\leq} 0} X)^i =
X^i$ for $i < 0$ , $({\tau}_{{\leq} 0} X)^0 = {Ker} d^0$ and zero
otherwise. Set ${\tau}_{{\geq} 1} X = {X/{{\tau}_{{\leq} 0} X}}.$ By
the assumptions on $A$, the canonical inclusion ${\tau}_{{\leq} 0} A
{\rightarrow} A$ is a quasi-isomorphism of dg algebras. Thus, we can
assume that $A^p$ is zero for all $p > 0$.

\begin{prop} [\cite{Am08}]\label{2}
Let $\mathcal{H}$ be the heart of the t-structure,
i.e.~$\mathcal{H}$ is the intersection ${\mathcal{D}}^{{\leq} {0}}
\cap {\mathcal{D}}^{{\geq} {0}}$.

a) The functor $H^0$ induces an equivalence from $\mathcal{H}$ onto
${\rm Mod} H^0 A$ .

b) For all X and Y in $\mathcal{H}$, we have an isomorphism
$${\rm Ext}_{H^0 A}^1 (X , Y) \simeq {\rm Hom}_{\mathcal{D}} (X , Y [1]).$$

\end{prop}

\begin{lem}\label{13}
For each integer r, the space $H^r A$ is finite-dimensional.
\end{lem}

\begin{proof}
By our assumptions, the space $H^r A$ is zero for every positive
integer $r$ and $H^0 A$ is finite-dimensional. We use induction on
$r$ to show that

{\it For each $r \geq m-1$, the space $H ^{m-r-1} A$ is
finite-dimensional, and for each $M \in {\rm mod} H^0 A$, the space
${\rm Hom}_{\mathcal{D}} ({\tau}_{{\leq} m-r-1} A , M [p])$ is
finite-dimensional for all $p {\geq} r+1$ .}

For $r = m - 1$, the space $H^0 A$ is finite-dimensional by
assumption. Since ${\tau}_{{\leq} 0} A$ is quasi-isomorphic to $A$,
the following isomorphisms hold
$${\rm Hom}_{\mathcal{D}} ({\tau}_{{\leq}
0} A , M [p]) \simeq {\rm Hom}_{\mathcal{D}} ( A , M [p]) \simeq H^0
(M [p]) = 0$$ for $p {\geq} (m-1+1) {\geq} 1$ and $M \in {\rm mod}
H^0 A$.

Suppose that the property holds for $r ( {\geq}{m-1})$. Applying the
functor ${\rm Hom}_{\mathcal{D}} (-, M [p])$ to the triangle :
$$(H^{m-r-1} A) [r-m] {\longrightarrow} {{\tau}_{\leq {m-r-2}} A} {\longrightarrow} {{\tau}_{\leq {m-r-1}} A} {\longrightarrow}
{(H^{m-r-1} A) [1+r-m]},$$ we can get a long exact
sequence$${\ldots}{\rightarrow}({\tau}_{{\leq} m-r-1} A, M[p])
{\rightarrow}({\tau}_{{\leq} m-r-2} A, M[p])
{\rightarrow}((H^{m-r-1} A) [r-m], M[p]) {\rightarrow} {\ldots},$$
where we write $(,)$ for ${\rm Hom}_{\mathcal{D}} (,)$.

By induction, the space ${\rm Hom}_{\mathcal{D}} ({\tau}_{{\leq}
m-r-1} A , M [p])$ is finite-dimensional for all $p {\geq} r+1$.
Note that $M [p]$ is in ${\mathcal{D}}^b (A)$, so we can use the
$(m+2)$-Calabi-Yau property of ${\mathcal{D}}$$^b (A)$.

If $p {\geq} r+3,$ then the space
$${\rm
Hom}_{\mathcal{D}} ( (H^{m-r-1} A) [r-m], M [p]) \simeq D{\rm
Hom}_{\mathcal{D}} ( M [p], (H^{m-r-1} A) [r-m+m+2])$$
$$\quad\quad\quad\quad\quad\quad\quad\quad\quad\quad\quad \simeq D{\rm Hom}_{\mathcal{D}} ( M [p-r-2], H^{m-r-1}
A)$$ vanishes when $p-r-2 \,{\geq} 1$, since $M [p-r-2]$ belongs to
${\mathcal{D}}^{{\leq} 2+r-p} \subseteq {\mathcal{D}}^{{\leq} -1}$
while $H^{m-r-1} A$ belongs to $\mathcal{H}$.

If $p = r+2,$ since by induction $H^{m-r-1} A$ is
finite-dimensional, we have that $${\rm Hom}_{\mathcal{D}} (
(H^{m-r-1} A) [r-m], M [p]) \simeq D{\rm Hom}_{H^0 A} ( M, H^{m-r-1}
A)$$ is also finite-dimensional, where the second isomorphism comes
from proposition $\ref{2}$. Therefore, the space ${\rm
Hom}_{\mathcal{D}} ({\tau}_{{\leq} m-r-2} A , M [p])$ is
finite-dimensional for all $p {\geq} r+2$.

Now consider the following diagram

\[
\xymatrix@H=2pc{
 & & M[r+4]& \\
\tau_{\leq {m-r-3}} A \ar[r]& \tau_{\leq {m-r-2}} A \ar[r]
\ar[ru]^{{\rm Hom}-finite} & (H^{m-r-2} A)[2+r-m] \ar[r] \ar@{.>}[u]
& \tau_{\leq {m-r-3}} A [1] . \ar[lu]_{{\rm Hom}-finite} } \]

Since $r+4 {\geq} r+2$ and $r+3 {\geq} r+3$, the spaces ${\rm
Hom}_{\mathcal{D}}( {\tau}_{{\leq} m-r-2} A, M[r+4])$ and ${\rm
Hom}_{\mathcal{D}} ({\tau}_{{\leq} m-r-3}A[1], M [r+4]) (\simeq {\rm
Hom}_{\mathcal{D}} ( {\tau}_{{\leq} m-r-3} A, M [r+3]))$ are
finite-dimensional for each $M \in {\rm mod} H^0 A$. Thus, the space
$${\rm Hom}_{\mathcal{D}} ( (H^{m-r-2} A) [2+r-m], M [r+4]) \simeq
{\rm Hom}_{\mathcal{D}} ( H^{m-r-2} A, M [m+2])$$
$$\simeq D {\rm Hom}_{\mathcal{D}} ( M, H^{m-r-2} A) \simeq D {\rm
Hom}_{H^0 A} ( M, H^{m-r-2} A)$$ is finite-dimensional for each $M
\in {\rm mod} H^0 A$. In particular, if we choose $M = H^0 A$, we
find that $D H^{m-r-2} A$ is finite-dimensional. Therefore, the
space $H^{m-r-2} A$ is finite-dimensional.
\end{proof}

The subcategory of $({\mbox{per}} A)^{op} \times {\mbox{per}} A$
whose objects are the pairs $(X,Y)$ such that, the space ${\rm
Hom}_{\mathcal {D}} (X,Y)$ is finite-dimensional, is stable under
extensions and passage to direct factors. By lemma \ref{13}, the
space $H^r A ( \simeq {\rm Hom}_{\mathcal {D}} (A, A[r]))$ is
finite-dimensional. As a result, the following proposition holds.
\begin{prop} \label{10}
The category ${\rm per}A$ is {\rm Hom}-finite.
\end{prop}

\begin{lem} [\cite{Am08}]\label{6}
For each $X$ in ${\rm per}A$, there exist integers $N$ and $M$ such
that $X$ belongs to ${\mathcal{D}}^{{\leq} N} \cap \,
^{\perp}{{\mathcal{D}}^{{\leq} M}}$. Moreover, the t-structure on
$\mathcal{D}$ canonically restricts to ${\rm per}A$.
\end{lem}

An obvious remark here is that the first statement in lemma \ref{6}
has the following equivalent saying:\,{\it there exists a positive
integer $N_0$ such that $X$ belongs to ${\mathcal {D}}^{\leq n} \cap
{^{\perp}{\mathcal{D}}^{{\leq} {-n}}}$ for any $n \geq {N_0}$}.

\begin{prop}\label{11}
The category $\mathcal {C}$ is $(m+1)$-Calabi-Yau.
\end{prop}

\begin{proof}
Let $\mathcal {T}$ denote the category ${\rm per}A$. Let $\mathcal
{N}$ denote $\mathcal {D}$$^b (A)$, which is a thick subcategory of
$\mathcal {T}$. Because of the Calabi-Yau property, that is,
\begin{center} $D {\rm Hom}_{\mathcal{D}} (N, X) \simeq {\rm
Hom}_{\mathcal{D}} (X, N[m+2]) \quad {\mbox {for each}} \, \,  N \in
{\mathcal {D}}$$^b (A)$ and $X \in \mathcal {D}$,\end{center} there
is a bifunctorial non-degenerate bilinear form :
$$\beta_{N,X} : {\rm
Hom}_{\mathcal{D}} (N, X) \times {\rm Hom}_{\mathcal{D}} (X, N[m+2])
\longrightarrow k$$ Therefore, by section 1 in \cite{Am08}, there
exists a bifunctorial form :
$$\beta'_{X,Y} : {\rm Hom}_{\mathcal{C}} (X, Y) \times {\rm
Hom}_{\mathcal{C}} (Y, X[m+1]) \longrightarrow k \quad {\mbox {for}}
\, \, X , Y \in {\mathcal {C}}.$$ By lemma \ref{6}, the object $X$
belongs to ${^{\perp}{\mathcal {D}}^{\leq r}}$ for some integer $r$.
Thus, we obtain an injection $$0 \longrightarrow {\rm Hom}_{\mathcal
{D}} (X, Y) \longrightarrow {\rm Hom}_{\mathcal {D}} (X, {\tau_{>
r}} Y),$$ and the object $\tau_{>r}Y$ is in $ {\mathcal {D}}$$^b
(A)$. Since ${\rm per}A$ is Hom-finite by proposition \ref{10},
still using section 1 in \cite{Am08}, we can get that $\beta'_{X,Y}$
is non-degenerate. Therefore, we have
$$D {\rm Hom}_{\mathcal{C}} (X, Y) \simeq {\rm Hom}_{\mathcal{C}}
(Y, X[m+1]) \quad {\mbox {for}} \, \, X , Y \in {\mathcal {C}}.$$
Thus, the category $\mathcal {C}$ is $(m+1)$-Calabi-Yau.
\end{proof}

Let $\mathcal{F}$ be the full subcategory ${\mathcal{D}}^{{\leq}
{0}} \cap \, ^{\perp}{\mathcal{D}}^{{\leq} {-m-1}} \cap \, {\rm
per}A $ of ${\rm per}A$.

\begin{lem}\label{3}
For each object $X$ of $\mathcal{F}$, there exist $m$ triangles (
which are not unique in general )
$$P_1 {\longrightarrow} Q_0 {\longrightarrow} X {\longrightarrow}
P_1 [1] ,$$
$$P_2 {\longrightarrow} Q_1 {\longrightarrow} P_1
{\longrightarrow} P_2 [1] ,$$
$$\ldots \quad \ldots$$
$$P_m {\longrightarrow} Q_{m-1} {\longrightarrow} P_{m-1}
{\longrightarrow} P_m [1] ,$$ where $Q_0, Q_1, \ldots , Q_{m-1}$ and
$P_m$ are in add$A$.
\end{lem}

\begin{proof}
For each object $X$ in ${\rm per}A$, the following isomorphisms
$${\rm Hom}_{\mathcal{D}} (A , X) \simeq H^0 X \simeq {\rm Hom}_{H^0
A} (H^0 A , H^0 X)$$ is true. Therefore, we can find a morphism $Q_0
{\longrightarrow} X$, with $Q_0$ a free dg $A$-module, which induces
an epimorphism $H^0 Q_0 {\twoheadrightarrow} H^0 X$. Take $X$ in
${\mathcal{F}}$ and form a triangle $$P_1 {\longrightarrow} Q_0
{\longrightarrow} X {\longrightarrow} P_1 [1] .$$

\smallskip

Step 1. {\it The object $P_1$ is in ${\mathcal{D}}^{{\leq} {0}} \cap
\, ^{\perp}{\mathcal{D}}^{{\leq} {-m}} \cap {\rm per}A$.}

Since the objects $Q_0$ and $X$ are in ${\mathcal{D}}^{{\leq} 0}$,
$P_1$ is in ${\mathcal{D}}^{{\leq} 1}$. Moreover, we have a long
exact sequence $$\ldots \, {\rightarrow} H^0 Q_0
{\twoheadrightarrow} H^0 X {\rightarrow} H^1 P_1 {\rightarrow} H^1
Q_0 = 0 .$$ It follows that $H^1 P_1 = 0$. Thus, the object $P_1$
belongs to ${{\mathcal{D}}^{{\leq} 0}}$.

Let $Y$ be in ${\mathcal{D}}^{\leq -m}$. Consider the long exact
sequence
$$\ldots \, \longrightarrow {\rm Hom}_{\mathcal{D}} ( Q_0, Y)
\longrightarrow {\rm Hom}_{\mathcal{D}} ( P_1, Y) \longrightarrow
{\rm Hom}_{\mathcal{D}} ( X[-1], Y) \longrightarrow \ldots .$$ Since
$X$ belongs to $^{\perp}{\mathcal{D}}^{{\leq} {-m-1}}$ and $Y$ is in
${\mathcal{D}}^{{\leq} {-m}}$, the space ${\rm Hom}_{\mathcal{D}} (
X[-1], Y)$ vanishes. The object $Q_0$ is free and $H^0 Y$ is zero,
so the space ${\rm Hom}_{\mathcal{D}} ( Q_0, Y)$ also vanishes. Thus
$P_1$ belongs to $^{\perp}{\mathcal{D}}^{{\leq} {-m}}$.

Moreover, since per$A$ is closed under extensions in ${\mathcal
{D}}$, the object $P_1$ belongs to ${\rm per}A$. Thus, the object
$P_1$ belongs to ${\mathcal{D}}^{{\leq} {0}} \cap \,
^{\perp}{\mathcal{D}}^{{\leq} {-m}} \cap \, {\rm per}A .$ Similarly
as above, we can find a morphism $Q_1 \longrightarrow P_1$, with
$Q_1$ a free dg $A$-module, which induces an epimorphism $H^0 Q_1
{\twoheadrightarrow} H^0 P_1$. Then we form a triangle
$$P_2 {\longrightarrow} Q_1 {\longrightarrow} P_1 {\longrightarrow}
P_2 [1] .$$

\smallskip

Step 2. {\it For $1 \leq r \leq m$, the object $P_t$ is in
${\mathcal{D}}^{{\leq} {0}} \cap \, ^{\perp}{\mathcal{D}}^{{\leq}
{t-m-1}} \cap \, {\rm per}A$}.

By the same argument as in step 1, we obtain that the object $P_2$
is in ${\mathcal{D}}^{{\leq} {0}} \cap \,
^{\perp}{\mathcal{D}}^{{\leq} {1-m}} \cap \, {\rm per}A$.

In this way, we inductively construct $m$ triangles $$P_1
{\longrightarrow} Q_0 {\longrightarrow} X {\longrightarrow} P_1 [1]
,$$
$$P_2 {\longrightarrow} Q_1 {\longrightarrow} P_1
{\longrightarrow} P_2 [1] ,$$
$$\ldots \quad \ldots$$
$$P_m {\longrightarrow} Q_{m-1} {\longrightarrow} P_{m-1}
{\longrightarrow} P_m [1] ,$$ where $Q_0, Q_1, \ldots , Q_{m-1}$ are
free dg $A$-modules and $P_t$ belongs to ${\mathcal{D}}^{{\leq} {0}}
\cap \, ^{\perp}{\mathcal{D}}^{{\leq} {t-m-1}} \cap \, {\rm per}A,
{\mbox{for each}} 1 {\leq} t {\leq} m$.

The following two steps are quite similar to the proof of Lemma 2.10
in {\cite{Am08}}. However, for the convenience of the reader, we
give a complete proof.

\smallskip

Step 3. {\it $H^0 P_m$ is a projective $H^0 A$-module.}

Since $P_m$ belongs to ${\mathcal{D}}^{\leq 0}$, there exists a
triangle
$${\tau}_{\leq -1} P_m \longrightarrow P_m \longrightarrow H^0 P_m \longrightarrow ({\tau}_{\leq -1} P_m)[1].$$
Take an object $M$ in the heart $\mathcal {H}$, and consider the
long exact sequence $$\ldots {\longrightarrow} (({\tau}_{{\leq} -1}
P_m) [1], M [1]) {\longrightarrow} ( H^0 P_m, M [1])
{\longrightarrow} ( P_m, M [1]) {\longrightarrow} {\ldots} ,$$ where
we write $(\,,)$ for ${\rm Hom}_{\mathcal{D}} (\,,)$. The space
${\rm Hom}_{\mathcal{D}} ( ({\tau}_{{\leq} -1} P_m) [1], M [1])$
vanishes because ${\rm Hom}_{\mathcal{D}} ({\mathcal {D}}^{\leq -2},
{\mathcal {D}}^{\geq -1})$ is zero. Since $P_m$ belongs to
$^{\perp}{\mathcal{D}}^{{\leq} {-1}}$, the space ${\rm
Hom}_{\mathcal{D}} ( P_m, M [1])$ also vanishes. As a result, the
space
$${\rm Ext}^1_{\mathcal {H}} (H^0 P_m M) \simeq {\rm Hom}_{\mathcal{D}} ( H^0 P_m, M
[1])$$ is zero. Thus, $H^0 P_m$ is a projective $H^0 A$-module.

\smallskip

Step 4. {\it $P_m$ is isomorphic to an object in add A. }

From step 3, we deduce that it is possible to find an object $P$ in
$add A$ and a morphism $P \longrightarrow P_m$ such that $H^0 P$ and
$H^0 P_m$ are isomorphic. Then we form a new triangle $$E
\longrightarrow P \longrightarrow P_m \longrightarrow E [1].$$ Since
$P$ and $P_m$ are in ${\mathcal {D}}^{\leq 0}$, the object $E$ is in
${\mathcal {D}}^{\leq 1}$. Moreover, there is a long exact sequence
$$\ldots \longrightarrow H^0 E \longrightarrow H^0 P \simeq H^0 P_m
\longrightarrow H^1 E \longrightarrow H^1 P = 0 .$$ So $E$ is in
${\mathcal {D}}^{\leq 0}$. Since $P_m$ belongs to
$^{\perp}{\mathcal{D}}^{{\leq} {-1}}$, the space ${\rm
Hom}_{\mathcal {D}} (P_m, E[1])$ vanishes. Therefore, the object $P$
is isomorphic to the direct sum of $P_m$ and $E$. Then we have an
isomorphism
$$H^0 P \simeq H^0 P_m \oplus H^0 E.$$ We obtain that $H^0 E$ is
zero. As a consequence, there is no nonzero morphism from $P$ to
$E$, since $P$ is a free $A$-module. Therefore, $E$ is the zero
object and $P_m$ is isomorphic to $P$ which is an object in $add A$.
\end{proof}

Let $N$ be an object of ${\mathcal {F}}$. By lemma $\ref{3}$, there
are $m$-triangles related to the object $N$. Denote by $\nu$ the
Nakayama functor on mod $H^0 A$. Clearly, $\nu H^0 P_m$ and $\nu H^0
Q_{m-1}$ are injective $H^0 A$-modules. Let $M$ be the kernel of the
morphism ${\nu H^0 P_m} \longrightarrow \nu H^0 Q_{m-1}$. It lies in
the heart $\mathcal {H}$. Let $X = N [1]$.

\begin{lem} \label{4}
(1) There are isomorphisms of functors: $${\rm Hom}_{\mathcal {D}}
(-, X[1])|_{\mathcal {H}} \simeq {\rm Hom}_{\mathcal {D}} (-,
P_1[3])|_{\mathcal {H}} \simeq \ldots$$ $$\ldots \simeq {\rm
Hom}_{\mathcal {D}} (-, P_{m-1}[m+1])|_{\mathcal {H}} \simeq {\rm
Hom}_{\mathcal {H}} (-, M).$$

(2) There is a monomorphism of functors: $${\rm Ext}^1_{\mathcal
{H}} (-, M) \hookrightarrow {\rm Hom}_{\mathcal {D}} (-,
P_{m-1}[m+2])|_{\mathcal {H}}.$$
\end{lem}

\begin{proof}
Let $L$ be in $\mathcal {H}$. Let us prove part (1).

Step 1. {\it There is an isomorphism of functors:$${\rm
Hom}_{\mathcal {D}} (-, P_{m-1}[m+1])|_{\mathcal {H}} \simeq {\rm
Hom}_{\mathcal {H}} (-, M).$$} Applying ${\rm Hom}_{\mathcal {D}}
(L, -)$ to the $m$-th triangle
$$P_{m} \longrightarrow Q_{m-1}\longrightarrow P_{m-1}
\longrightarrow P_m [1], $$ we obtain a long exact sequence $$\ldots
\longrightarrow {\rm Hom}_{\mathcal {D}} (L, Q_{m-1}[m+1])
\longrightarrow {\rm Hom}_{\mathcal {D}} (L, P_{m-1}[m+1])
\longrightarrow $$ $$ \longrightarrow {\rm Hom}_{\mathcal {D}} (L,
P_{m}[m+2]) \longrightarrow {\rm Hom}_{\mathcal {D}} (L,
Q_{m-1}[m+2]) \longrightarrow \ldots .$$ Since $L$ belongs to
${\mathcal {D}}$$^b (A)$, by the Calabi-Yau property one can easily
see the following isomorphism
$${\rm Hom}_{\mathcal {D}} (L, Q_{m-1}[m+1]) \simeq D {\rm
Hom}_{\mathcal {D}} (Q_{m-1}, L[1]).$$ The space vanishes since the
object $Q_{m-1}$ is a free dg $A$-module and $H^1 L$ is zero.
Consider the triangle
$${\tau}_{\leq -1} P_m \longrightarrow P_m \longrightarrow H^0 P_m
\longrightarrow ({\tau}_{\leq -1} P_m)[1].$$ We can get a long exact
sequence
$$\ldots {\longrightarrow} (({\tau}_{{\leq} -1} P_m) [1], L) {\longrightarrow}
( H^0 P_m, L) {\longrightarrow} ( P_m, L) {\longrightarrow}
({\tau}_{{\leq} -1} P_m, L) \longrightarrow {\ldots}\ ,$$ where we
write $(,)$ for ${\rm Hom}_{\mathcal{D}} (,)$. Since ${\rm
Hom}_{\mathcal{D}} ({\mathcal {D}}^{\leq {-1-i}}, {\mathcal
{D}}^{\geq 0})$ is zero, the space ${\rm Hom}_{\mathcal{D}} (
({\tau}_{{\leq} -1} P_m) [i], L)$ vanishes for $i = 0, 1$. Thus, we
have
$${\rm Hom}_{\mathcal{D}} ( P_m, L) \simeq {\rm Hom}_{\mathcal{D}} (
H^0 P_m, L) \simeq {\rm Hom}_{\mathcal{H}} (H^0 P_m, L).$$ Combining
with the Calabi-Yau property, we get the following isomorphisms
$${\rm Hom}_{\mathcal {D}} (L, P_m [m+2]) \simeq D {\rm
Hom}_{\mathcal {D}}(P_m, L)$$ $$ \quad
\quad\quad\quad\quad\quad\quad \simeq D {\rm Hom}_{\mathcal{H}} (H^0
P_m, L) \simeq {\rm Hom}_{\mathcal{H}} (L, \nu H^0 P_m).$$
Similarly, we can see that $${\rm Hom}_{\mathcal{D}} (L, Q_{m-1}
[m+2]) \simeq {\rm Hom}_{\mathcal{H}} (L, \nu H^0 Q_{m-1}).$$
Therefore, the functor ${\rm Hom}_{\mathcal {D}} (-,
P_{m-1}[m+1])|_{\mathcal {H}}$ is isomorphic to the functor ${\rm
Hom}_{\mathcal {H}} (-, M)$, which is the kernel of the morphism
$${\rm Hom}_{\mathcal{H}} (-, \nu H^0 P_m) \longrightarrow {\rm
Hom}_{\mathcal{H}} (-, \nu H^0 Q_{m-1}).$$

\smallskip

Step 2. {\it There are isomorphisms of functors: $${\rm
Hom}_{\mathcal {D}} (-, X[1])|_{\mathcal {H}} \simeq {\rm
Hom}_{\mathcal {D}} (-, P_1[3])|_{\mathcal {H}} \simeq \ldots \simeq
{\rm Hom}_{\mathcal {D}} (-, P_{m-1}[m+1])|_{\mathcal {H}}.$$}
Applying the functor ${\rm Hom}_{\mathcal {D}} (L, -)$ to the
$(m-1)$-th triangle
$$P_{m-1} \longrightarrow Q_{m-2}\longrightarrow P_{m-2}
\stackrel{h_{m-2}}\longrightarrow P_{m-1} [1], $$ we obtain a long
exact sequence
$$\ldots \longrightarrow {\rm Hom}_{\mathcal {D}} (L, Q_{m-2}[m])
\longrightarrow {\rm Hom}_{\mathcal {D}} (L, P_{m-2}[m])
\longrightarrow $$ $$ \longrightarrow {\rm Hom}_{\mathcal {D}} (L,
P_{m-1}[m+1]) \longrightarrow {\rm Hom}_{\mathcal {D}} (L,
Q_{m-2}[m+1]) \longrightarrow \ldots .$$ Since $Q_{m-2}$ is a free
$A$-module and $L$ is in $\mathcal {H}$, the space ${\rm
Hom}_{\mathcal {D}} (Q_{m-2}, L[r])$ vanishes for each positive
integer $r$. As a result, by the Calabi-Yau property, the following
two isomorphisms hold
$${\rm Hom}_{\mathcal {D}} (L, Q_{m-2}[m]) \simeq D {\rm
Hom}_{\mathcal {D}} (Q_{m-2}, L[2]) = 0, $$ $${\rm Hom}_{\mathcal
{D}} (L, Q_{m-2}[m+1]) \simeq D {\rm Hom}_{\mathcal {D}} (Q_{m-2},
L[1]) = 0.$$ Therefore, we have $${\rm Hom}_{\mathcal {D}} (-,
P_{m-2}[m])|_{\mathcal {H}} \simeq {\rm Hom}_{\mathcal {D}} (-,
P_{m-1}[m+1])|_{\mathcal {H}},$$ where the isomorphism is induced by
the left multiplication by $h_{m-2} [m]$.

We inductively work with each triangle and get a corresponding
isomorphism induced by the left multiplication by $h_{m-r} [m-r+2]$,
$${\rm Hom}_{\mathcal {D}} (-, P_{m-r}[m-r+2])|_{\mathcal {H}}
\simeq {\rm Hom}_{\mathcal {D}} (-, P_{m-r+1}[m-r+3])|_{\mathcal
{H}}, \, 2 \leq r \leq {m-1},$$ while the isomorphism $${\rm
Hom}_{\mathcal {D}} (-, X[1])|_{\mathcal {H}} \simeq {\rm
Hom}_{\mathcal {D}} (-, P_1[3])|_{\mathcal {H}}$$ is induced by the
left multiplication by $h_0 [2]$. Therefore, the first assertion in
this lemma holds.

\newpage

Let us prove part (2).

Consider the following long exact sequence
$$\ldots \longrightarrow {\rm Hom}_{\mathcal {D}} (L, P_{m}[m+2])
\longrightarrow {\rm Hom}_{\mathcal {D}} (L, Q_{m-1}[m+2])
\longrightarrow $$
$$\longrightarrow {\rm Hom}_{\mathcal {D}} (L, P_{m-1}[m+2])
\longrightarrow {\rm Hom}_{\mathcal {D}} (L,
P_{m}[m+3])\longrightarrow \ldots .$$ By the Calabi-Yau property,
the space ${\rm Hom}_{\mathcal {D}} (L, P_{m}[m+3])$ is isomorphic
to the zero space $D {\rm Hom}_{\mathcal {D}} (P_{m}[1],L)$.

Hence, the functor ${\rm Hom}_{\mathcal {D}} (-,
P_{m-1}[m+2])|_{\mathcal {H}}$ is isomorphic to the cokernel of the
morphism $${\rm Hom}_{\mathcal{H}} (-, \nu H^0 P_m) \longrightarrow
{\rm Hom}_{\mathcal{H}} (-, \nu H^0 Q_{m-1}).$$ As an $H^0
A$-module, $M$ admits an injective resolution of the following form
$$0 \longrightarrow {\nu H^0 P_m} \longrightarrow {\nu H^0 Q_{m-1}
\longrightarrow I \longrightarrow \ldots ,}$$ where $I$ is an
injective $H^0 A$-module. Then ${\rm Ext}^1_{\mathcal {H}} (-, M)$
is the first homology of the following complex
$$0 \longrightarrow {\rm Hom}_{\mathcal{H}}
(-, \nu H^0 P_m) \longrightarrow {\rm Hom}_{\mathcal{H}} (-, \nu H^0
Q_{m-1}) \longrightarrow {\rm Hom}_{\mathcal {H}} (-, I)
\longrightarrow \ldots .$$ Therefore, we get a monomorphism of
functors $${\rm Ext}^1_{\mathcal {H}} (-, M) \hookrightarrow {\rm
Hom}_{\mathcal {D}} (-, P_{m-1}[m+2])|_{\mathcal {H}}.$$
\end{proof}

Following Step 1 in the proof of lemma \ref{4}, there is an
isomorphism of functors: $${\rm Hom}_{\mathcal {D}} (-,
P_{m-1}[m+1])|_{\mathcal {H}} \simeq {\rm Hom}_{\mathcal {H}} (-,
M).$$ We denote it by ${\varphi}_1$, and when ${\varphi}_1$ is
applied to an object $X$ in $\mathcal {H}$, we denote the
isomorphism by ${\varphi}_{1,X}$. Let ${\rho}$ be the preimage of
the identity map on $M$ under the isomorphism
$${\varphi}_{1,M}: {\rm Hom}_{\mathcal {D}} (M, P_{m-1}[m+1]) \simeq
{\rm Hom}_{\mathcal {H}} (M, M).$$ Now we can form a triangle
$$P_{m-1} [m] \longrightarrow Y' \longrightarrow M
\stackrel{\rho}\longrightarrow P_{m-1} [m+1].$$

\begin{lem} \label{5}
The object $Y'$ is in $\mathcal {F}$.
\end{lem}

\begin{proof}
Since $M$ belongs to ${\mathcal {H}}$ and $P_{m-1}[m+1]$ belongs to
${\rm per}A$, it follows that $Y'$ is also in ${\rm per}A$.
Moreover, $Y'$ is in ${\mathcal {D}}^{\leq 0}$, since the objects
$M$ and $P_{m-1}$ are in ${\mathcal {D}}^{\leq 0}$. Let $Z$ be an
object in ${\mathcal{D}}^{{\leq} {-m-1}}$. Then there is a long
exact sequence
$$\ldots \, \longrightarrow {\rm Hom}_{\mathcal{D}} ( P_{m-1} [m+1],
Z) \longrightarrow {\rm Hom}_{\mathcal{D}} ( M, Z) \longrightarrow
$$ $$ \longrightarrow {\rm Hom}_{\mathcal{D}} ( Y', Z)
\longrightarrow {\rm Hom}_{\mathcal{D}} ( P_{m-1} [m], Z)
\longrightarrow {\rm Hom}_{\mathcal{D}} ( M [-1], Z) \longrightarrow
\ldots .$$ Since $Z$ belongs to ${\mathcal{D}}^{{\leq} {-m-1}}$, we
have the following triangle
$${\tau}_{\leq -m-2} Z \longrightarrow Z \longrightarrow
(H^{-m-1} Z) [m+1] \longrightarrow ({\tau}_{\leq -m-2} Z)[1].$$ By
the Calabi-Yau property, the space
$${\rm Hom}_{\mathcal{D}} ( M[-1],
({\tau}_{{\leq} {-m-2}} Z) [i]) \simeq D {\rm Hom}_{\mathcal{D}} (
{\tau}_{{\leq} {-m-2}} Z, M [m+1-i])$$ is zero for $i = 0, 1.$ As a
result, we have that $${\rm Hom}_{\mathcal{D}} ( M [-1], Z) \simeq
{\rm Hom}_{\mathcal{D}} ( M [-1], (H^{-m-1} Z) [m+1])$$
$$\quad\quad\quad \simeq D {\rm Hom}_{\mathcal{D}} ( H^{-m-1} Z,
M).$$

From Step 2 in the proof of lemma \ref{3}, we know that the object
$P_{m-1}$ is in $^{\perp}{\mathcal{D}}^{{\leq} {-2}}$. So the $m$-th
shift $P_{m-1} [m]$ is in $^{\perp}{\mathcal{D}}^{{\leq} {-m-2}}$.
Combining with the Calabi-Yau property, the following isomorphisms
$${\rm Hom}_{\mathcal{D}} ( P_{m-1} [m], Z) \simeq {\rm
Hom}_{\mathcal{D}} ( P_{m-1} [m], (H^{-m-1} Z) [m+1])$$
$$\quad\quad\quad\quad\quad\quad\quad \simeq D {\rm Hom}_{\mathcal{D}} ( H^{-m-1} Z, P_{m-1}
[m+1])$$ hold. Now by lemma {\ref{4}}, we obtain an isomorphism
$${\rm Hom}_{\mathcal{D}} ( P_{m-1} [m], Z) \simeq {\rm
Hom}_{\mathcal{D}} ( M [-1], Z).$$ Consider the following
commutative diagram

\[
\xymatrix{ (P',Z') \ar[r] \ar[d]^a & (P',Z) \ar[r] \ar[d]^b &
(P',(H^{-m-1}Z)[m+1]) \ar[r] \ar[d]^c & (P',Z'[1])
\ar[d]^d \\
(M,Z') \ar[r] & (M,Z) \ar[r]& (M,(H^{-m-1}Z)[m+1]) \ar[r] &
(M,Z'[1]), }
\]
where we write $(,)$ for ${\rm Hom}_{\mathcal {D}} (,)$, $P'$ for
$P_{m-1}[m+1]$, and $Z'$ for $\tau_{\leq {-m-2}}Z$. Since the object
$P_{m-1} [m+1]$ is in $^{\perp}{\mathcal{D}}^{{\leq} {-m-3}}$, we
have that the space ${\rm Hom}_{\mathcal{D}} ( P_{m-1} [m+1],
({\tau}_{{\leq} {-m-2}} Z) [1])$ vanishes, and then the rightmost
morphism $d$ is a zero map. By the Calabi-Yau property and
proposition \ref{2}, on can easily get the following isomorphisms
$${\rm Hom}_{\mathcal{D}} ( P_{m-1} [m+1], (H^{-m-1} Z) [m+1])
\simeq D {\rm Hom}_{\mathcal{D}} ( H^{-m-1} Z, P_{m-1} [m+2]),$$
\begin{flushleft}
${\rm Hom}_{\mathcal{D}} ( M, (H^{-m-1} Z) [m+1]) \simeq D {\rm
Hom}_{\mathcal{D}} ( H^{-m-1} Z, M[1])$
\end{flushleft}
$\quad\quad\quad\quad\quad\quad\quad\quad\quad\quad\quad\quad\quad
\simeq D {\rm Ext}^1_{\mathcal {H}} (H^{-m-1} Z, M).$

\noindent Then by lemma \ref{4}, the morphism $c$ is surjective.
Consider the triangle
$${\tau}_{\leq -m-3} Z \longrightarrow {\tau}_{\leq -m-2} Z \longrightarrow
(H^{-m-2} Z) [m+2] \longrightarrow ({\tau}_{\leq -m-3} Z)[1].$$
Applying the functor ${\rm Hom}_{\mathcal {D}} ( - , M[m+2])$ to
this triangle and by the Calabi-Yau property, we can obtain
isomorphisms as follows:
$${\rm Hom}_{\mathcal {D}} (M , {\tau}_{\leq -m-2} Z ) \simeq D
{\rm Hom}_{\mathcal {D}} ({\tau}_{\leq -m-2} Z , M[m+2])$$ $$ \simeq
D {\rm Hom}_{\mathcal {D}} ((H^{-m-2} Z) [m+2], M[m+2]) \simeq D
{\rm Hom}_{\mathcal {D}} (H^{-m-2} Z, M).$$ Applying the functor
${\rm Hom}_{\mathcal {D}} (P_{m-1} [m+1], -)$ to the same triangle
and by the Calabi-Yau property, we can get isomorphisms as follows:
$${\rm Hom}_{\mathcal {D}} (P_{m-1} [m+1], {\tau}_{\leq -m-2} Z) \simeq
{\rm Hom}_{\mathcal {D}} ( P_{m-1} [m+1], (H^{-m-2} Z) [m+2])$$
$$\simeq D{\rm Hom}_{\mathcal {D}} (H^{-m-2} Z, P_{m-1} [m+1]).$$
Therefore, following lemma \ref{4}, the leftmost morphism $a$ is an
isomorphism. Then by Five-Lemma, the morphism $b$ is surjective.
From the long exact sequence at the beginning of the proof, we can
see that the space ${\rm Hom}_{\mathcal {D}} (Y' , Z)$ vanishes for
any $Z \in {\mathcal{D}}^{{\leq} {-m-1}}$. Hence, the object $Y'$ is
in $\mathcal {F}$.
\end{proof}

Let $\varphi_r (2 \leq r \leq {m-1})$ denote the isomorphism
$${\rm Hom}_{\mathcal {D}} (-, P_{m-r} [m-r+2])|_{\mathcal {H}}
\simeq {\rm Hom}_{\mathcal {D}} (-, P_{m-r+1}[m-r+3])|_{\mathcal
{H}}$$  in lemma \ref{3}, and let $\varphi_m$ denote the isomorphism
$${\rm Hom}_{\mathcal {D}} (-, X[1])|_{\mathcal {H}} \simeq {\rm
Hom}_{\mathcal {D}} (-, P_1[3])|_{\mathcal {H}}.$$ We write $f$ for
the composition $h_{m-2}[m] \ldots h_0 [2]$, and $\theta$ for the
composition $\varphi_1 \ldots \varphi_m$. Let $\varepsilon$ be the
preimage of the identity map on $M$ under the isomorphism
$$\theta_M : {\rm Hom}_{\mathcal {D}} (M, X[1]) \simeq {\rm
Hom}_{\mathcal {H}} (M, M).$$ As a result, we have that $$\theta_M
(\varepsilon) = id_M = \varphi_{1 M}({\rho}).$$ Thus, the following
equalities hold, $$f \varepsilon = \varphi_{2,M} \ldots
\varphi_{m,M} (\varepsilon) = \rho .$$ Now we form a new triangle
$$X \longrightarrow Y \longrightarrow M \stackrel{\varepsilon}\longrightarrow X[1].$$
\begin{lem} \label{7}
The object $Y$ is in $\mathcal {F}$ and $\tau_{\leq {-1}} Y$ is
isomorphic to $X$.
\end{lem}

\begin{proof}
Since $M$ belongs to ${\mathcal {H}}$ and $X$ belongs to ${\mathcal
{D}}^{\leq 0} \cap {\rm per}A$, the object $Y$ is also in ${\mathcal
{D}}^{\leq 0} \cap {\rm per}A$. Our aim is to show that $Y$ is in
$^{\perp}{\mathcal{D}}^{{\leq} {-m-1}}$. Let $Z$ be an object in
${\mathcal{D}}^{{\leq} {-m-1}}$. There is a long exact sequence
$$\ldots \, \longrightarrow {\rm Hom}_{\mathcal{D}} (
X[1], Z) \longrightarrow {\rm Hom}_{\mathcal{D}} ( M, Z)
\longrightarrow $$ $$ \longrightarrow {\rm Hom}_{\mathcal{D}} ( Y,
Z) \longrightarrow {\rm Hom}_{\mathcal{D}} ( X, Z) \longrightarrow
{\rm Hom}_{\mathcal{D}} ( M [-1], Z) \longrightarrow \ldots .$$
Since $X$ is in $^{\perp}{\mathcal{D}}^{{\leq} {-m-2}}$, by the same
argument as in lemma \ref{5}, we can obtain an isomorphism
$${\rm Hom}_{\mathcal{D}} ( X, Z) \simeq {\rm Hom}_{\mathcal{D}} ( M
[-1], Z) .$$ Since $\rho$ is the composition $f \varepsilon$, there
exists a morphism $g : Y \longrightarrow Y'$ such that the following
diagram is commutative
\[
\xymatrix{ X \ar[r] \ar[d]^{f[-1]}& Y \ar[r] \ar@{.>}[d]^g &M \ar[r]^{\varepsilon} \ar@{=}[d] &X[1] \ar[d]^f \\
P_{m-1}[m] \ar[r] & Y' \ar[r] & M \ar[r]^-{\rho} &
\,P_{m-1}[m+1].}\]

Applying the functor ${\rm Hom}_{\mathcal {D}} (- , Z)$ to this
diagram, then we get the following commutative diagram
\[
\xymatrix{ (M,Z) \ar[r] \ar@{=}[d] & (Y',Z) \ar[r] \ar[d]^-{{\rm
Hom}_{\mathcal {D}} (g,-)|_Z} &
(P_{m-1}[m],Z) \ar[r] \ar[d]^-{{\rm Hom}_{\mathcal {D}} (f[-1],-)|_Z} & (M[-1], Z) \ar@{=}[d]\\
(M,Z) \ar[r] & (Y,Z) \ar[r] & (X,Z) \ar[r] & (M[-1],Z), }
\]
where we write $(,)$ for ${\rm Hom}_{\mathcal {D}} (,)$. The
morphism $${\rm Hom}_{\mathcal {D}} (f [-1],-)|_Z : {\rm
Hom}_{\mathcal {D}} (P_{m-1}[m], Z) \longrightarrow {\rm
Hom}_{\mathcal{D}} ( X, Z)$$ is an isomorphism. We can see this as
follows:

applying the functor ${\rm Hom}_{\mathcal {D}} ( -, Z)$ to triangles
$$P_1[1] \longrightarrow Q_0[1] \longrightarrow X
\stackrel{h_0[1]}\longrightarrow P_1[2], \quad {\mbox{and}}$$
$$P_r[r] \longrightarrow Q_{r-1}[r] \longrightarrow P_{r-1}[r]
\stackrel{h_{r-1}[r]}\longrightarrow P_r[r+1], \quad 2 \leq r \leq
{m-1},$$ we can get long exact sequences (here we denote $X$ by
$P_0[1]$)
$$\ldots \longrightarrow {\rm Hom}_{\mathcal {D}} ( Q_{r-1}[r+1], Z)
\longrightarrow {\rm Hom}_{\mathcal {D}} ( P_r[r+1], Z)
\longrightarrow $$ $$ \longrightarrow {\rm Hom}_{\mathcal {D}} (
P_{r-1}[r], Z) \longrightarrow {\rm Hom}_{\mathcal {D}} (
Q_{r-1}[r], Z) \longrightarrow \ldots , \quad  1 \leq r \leq
{m-1}.$$ The objects $Z [-r-i] (i = 0, 1)$ are in ${\mathcal
{D}}^{\leq {r+i-m-1}} (\subset {\mathcal {D}}^{\leq {-1}})$. Since
$Q_{r-1}$ is a free $A$-module, the space ${\rm Hom}_{\mathcal {D}}
( Q_{r-1}[r+i], Z)$ vanishes for $i=0, 1$. Thus, the morphism $${\rm
Hom}_{\mathcal {D}} (h_{r-1}[r],-)|_Z: {\rm Hom}_{\mathcal {D}} (
P_{r}[r+1], Z) \simeq {\rm Hom}_{\mathcal {D}} ( P_{r-1}[r], Z)$$ is
an isomorphism for each $1 \leq r \leq {m-1}.$ As a consequence, the
functor ${\rm Hom}_{\mathcal {D}} (f [-1],-)|_Z$ is an isomorphism.
By Five-Lemma, we can obtain that ${\rm Hom}_{\mathcal {D}}
(g,-)|_Z$ is an epimorphism. From lemma \ref{5}, we know that the
object $Y'$ is in $\mathcal {F}$, and the space ${\rm Hom}_{\mathcal
{D}} (Y', Z)$ vanishes. It follows that the space ${\rm
Hom}_{\mathcal {D}} (Y, Z)$ is also zero, hence $Y$ is in $\mathcal
{F}$.

Since $X$ is in ${\mathcal {F}} [1]$, the spaces $H^0 X$ and $H^1 X$
are zero. Thus, the object $H^0 Y$ is isomorphic to $M$. Moreover,
the space ${\rm Hom}_{\mathcal {D}} (X, H^0 Y)$ is zero. Hence, we
can obtain a commutative diagram of triangles
\[
\xymatrix{ \tau_{\leq {-1}} Y \ar[r] & Y \ar[r]^{p_Y} \ar@{=}[d] &
H^0 Y \ar[r] &
(\tau_{\leq {-1}} Y)[1] \\
X \ar[r] \ar@{.>}[u]_{\delta_2} & Y \ar[r] & M \ar[r]
\ar@{.>}[u]_{\delta_1} & X[1] , \ar@{.>}[u]}
\]
where ${\delta_1}: M \longrightarrow H^0 Y$ is an epimorphism
between isomorphic terms. Therefore, $\delta_1$ is an isomorphism.
Thus, $\tau_{\leq {-1}} Y$ is isomorphic to $X$.

\end{proof}

\begin{lem} \label{8}
The image of the functor ${\tau}_{\leq {-i}}$ restricted to
$\mathcal {F}$ is in ${\mathcal {F}}[i]$ and the functor
${\tau}_{\leq {-i}}: {\mathcal {F}} \longrightarrow {\mathcal {F}}
[i]$ is fully faithful for each positive integer $i$.
\end{lem}

\begin{proof}
Let $X$ be an object in $\mathcal {F}$. Then ${\tau}_{\leq {-i}} X$
is in ${\mathcal {D}}^{\leq {-i}}$, and there is a triangle in
$\mathcal {D}$
$${\tau}_{\leq {-i}} X \longrightarrow X \longrightarrow
{\tau}_{> {-i}} X \longrightarrow ({\tau}_{\leq {-i}} X) [1].$$
Following lemma \ref{6}, the object ${\tau}_{\leq {-i}} X$ belongs
to ${\mathcal {D}}^{\leq {-i}} \cap {\rm per}A$. Let $Y$ be an
object in ${\mathcal {D}}^{\leq {-m-i-1}}$. Applying the functor
${\rm Hom}_{\mathcal {D}} (-, Y)$ to this triangle, then we can get
a long exact sequence
$$\ldots \rightarrow {\rm Hom}_{\mathcal {D}} (X, Y)
\rightarrow {\rm Hom}_{\mathcal {D}} ({\tau}_{\leq {-i}} X, Y)
\rightarrow {\rm Hom}_{\mathcal {D}} (({\tau}_{> {-i}} X)[-1], Y)
\rightarrow \ldots .$$ The space ${\rm Hom}_{\mathcal {D}} (X, Y)$
vanishes because $X$ is in $^{\perp}{\mathcal{D}}^{{\leq} {-m-1}}$
and $i$ is a positive integer. Since ${\tau}_{> {-i}} X$ is in
${\mathcal {D}}$$^b (A)$, by the Calabi-Yau property, we have the
isomorphism
$${\rm Hom}_{\mathcal {D}} (({\tau}_{> {-i}} X)[-1], Y) \simeq D
{\rm Hom}_{\mathcal {D}} ( Y, ({\tau}_{> {-i}} X)[m+1]) = 0 .$$
Hence the space ${\rm Hom}_{\mathcal {D}} ({\tau}_{\leq {-i}} X, Y)$
is also zero. It follows that ${\tau}_{\leq {-i}} X$ belongs to
${\mathcal {F}}[i]$.

Let $X, Y$ be two objects in $\mathcal {F}$ and $f: {\tau}_{\leq
{-i}} X \longrightarrow {\tau}_{\leq {-i}} Y$ a morphism. Consider
the following diagram
\[
\xymatrix{(\tau_{>{-i}}X)[-1] \ar[r] & \tau_{\leq {-i}}X
\ar[r]^{s_X^i} \ar[d]^f & X \ar[r] \ar@{.>}[d] & \tau_{> {-i}}X \\
(\tau_{>{-i}}Y)[-1] \ar[r] & \tau_{\leq {-i}}Y \ar[r]^{s_Y^i} & Y
\ar[r] & \tau_{> {-i}}Y .}
\]

For $j = 0, 1$, by the Calabi-Yau property, the isomorphism holds
$${\rm Hom}_{\mathcal {D}} (({\tau}_{> {-i}} X)[-j], Y) \simeq D {\rm Hom}_{\mathcal {D}} (Y, ({\tau}_{> {-i}}
X)[m+2-j]) = 0 ,$$ since $Y$ is an object in
$^{\perp}{\mathcal{D}}^{{\leq} {-m-1}}$.

Since the space ${\rm Hom}_{\mathcal {D}} (({\tau}_{> {-i}} X)[-1],
Y)$ vanishes, the composition $s_Y^i f$ factors through $s_X^i$.
Thus, the functor ${\tau}_{\leq {-i}}$ is full.

Let $g: X \longrightarrow Y$ be a morphism in $\mathcal {F}$
satisfying $\tau_{\leq {-i}} g$ is zero. Then it induces a
commutative diagram as follows
\[
\xymatrix{\tau_{\leq {-i}}X \ar[r]^{s_X^i} \ar[d]|{\tau_{\leq
{-i}}g} & X \ar[r]^{p_X^i} \ar[d]^g &
\tau_{>{-i}}X \ar[r] \ar@{.>}[ld]^{g_1} & (\tau_{\leq {-i}}X)[1] \\
\tau_{\leq {-i}}Y \ar[r]^{s_Y^i} & Y \ar[r]^{p_Y^i} & \tau_{>
 {-i}}Y \ar[r] & (\tau_{\leq {-i}}X)[1] .
}\]

\noindent The morphism $g s_X^i$ is zero, so the morphism $g$ factor
through $p_X^i$. That is, there exists a morphism $g_1 : {\tau}_{>
{-i}} X \longrightarrow Y$ such that $g = g_1 p_X^i$. The morphism
$g_1$ is zero, since the space ${\rm Hom}_{\mathcal {D}} ({\tau}_{>
{-i}} X, Y)$. Thus, the morphism $g$ is zero. It follows that the
functor ${\tau}_{\leq {-i}}$ is faithful. Now this lemma holds.
\end{proof}

Together by lemma \ref{7} and lemma \ref{8}, we know that the
functor ${\tau}_{\leq {-1}} : {\mathcal {F}} \longrightarrow
{\mathcal {F}}[1]$ is an equivalence.

By the same arguments as Step 1 and Step 2 in the proof of
proposition 2.9 in \cite{Am08}, we can get the following two lemmas.
However, for the convenience of our later proposition \ref{12}, we
would like to write down the proof of the second lemma, which
presents a procedure of constructing the needed object.

\begin{lem}
The functor $\pi$ (restricted to $\mathcal {F}$)$: \mathcal {F}
\longrightarrow \mathcal {C}$ is fully faithful.
\end{lem}

\begin{lem}\label{9}
For any object $X$ in ${\rm per}A$, there exists an integer $r$ and
an object $Z$ in ${\mathcal {F}} [-r]$ such that $\pi X$ and $\pi Z$
are isomorphic objects in the category $\mathcal {C}$.
\end{lem}

\begin{proof}
Let $X$ be an object in ${\rm per}A$. By lemma \ref{6}, there exists
an integer $r$ such that $X$ is in ${\mathcal{D}}^{{\leq} {m+1-r}}
\cap {^{\perp}{\mathcal{D}}^{{\leq} {r-m-1}}}$. Consider the
triangle
$${\tau}_{\leq {r}} X \longrightarrow X
\longrightarrow {\tau}_{> {r}} X \longrightarrow ({\tau}_{\leq {r}}
X) [1].$$ Let $Y$ be an object in ${\mathcal{D}}^{{\leq} {r-m-1}}$.
Applying the functor ${\rm Hom}_{\mathcal {D}} (-, Y)$, we can get a
long exact sequence $$\ldots \rightarrow {\rm Hom}_{\mathcal {D}}
(X, Y) \longrightarrow {\rm Hom}_{\mathcal {D}} ({\tau}_{\leq {r}}
X, Y) \longrightarrow {\rm Hom}_{\mathcal {D}} (({\tau}_{> {r}}
X)[-1], Y) \rightarrow \ldots .$$ Clearly, the space ${\rm
Hom}_{\mathcal {D}} (X, Y)$ is zero. By the Calabi-Yau property, we
have the isomorphism $${\rm Hom}_{\mathcal {D}} (({\tau}_{> {r}}
X)[-1], Y) \simeq D {\rm Hom}_{\mathcal {D}} ( Y, ({\tau}_{> {r}}
X)[m+1]) = 0 .$$ Therefore, the object ${\tau}_{\leq r} X$ is in
$^{\perp}{\mathcal{D}}^{{\leq} {r-m-1}}$. Thus, we have that
${\tau}_{\leq r} X$ is in ${\mathcal {F}} [-r]$. Let $Z$ denote
${\tau}_{\leq r} X$. Since ${\tau}_{> r} X$ is in ${\mathcal
{D}}$$^b (A)$, the objects $\pi X$ and $\pi Z$ are isomorphic in
$\mathcal {C}$.
\end{proof}

\begin{prop}\label{12}
The projection functor $\pi: {\rm per}A \longrightarrow \mathcal
{C}$ induces a $k$-linear equivalence between $\mathcal {F}$ and
$\mathcal {C}$.
\end{prop}

\begin{proof}
We only need to show that $\pi$ restricted to $\mathcal {F}$ is
dense. Let $X$ be an object in ${\rm per}A$. Then there exists an
integer $r$ such that, the object $X$ is in ${\mathcal{D}}^{{\leq}
{m+1-r}} \cap {^{\perp}{\mathcal{D}}^{{\leq} {r-m-1}}}$, the object
${\tau_{\leq {r}}} X$ is in ${\mathcal {F}}[-r]$, and $\pi X$ is
isomorphic to $\pi ({\tau_{\leq r}} X)$ in $\mathcal {C}$. Now we do
induction on the number $r$. From the remark of lemma \ref{6}, we
can suppose that $r \leq 0$.

If $r=0$, the object $\tau_{\leq 0}X$ is in $\mathcal {F}$, and $\pi
(\tau_{\leq 0}X)$ is isomorphic to the image $\pi X$of $X$ in
$\mathcal {C}$.

Suppose when $r = r_0 \leq 0$, one can find an object $Y$ in
$\mathcal {F}$ such that $\pi Y$ is isomorphic to $\pi X$ in
$\mathcal {C}$.

Consider the case $r = r_0 - 1$. Then ${\tau_{\leq {r_0 -1}}} X$ is
in ${\mathcal {F}}[1 - r_0]$. Set $Z = ({\tau_{\leq {r_0 -1}}} X)
[-1]$. Thus, the object $Z$ is in ${\mathcal {F}} [-r_0]$. By
hyperthese, there exists an object $Y$ in $\mathcal {F}$ such that
$\pi Y$ is isomorphic to $\pi Z$ in $\mathcal {C}$. Therefore, we
have following isomorphisms in $\mathcal {C}$
$$\pi Y \simeq \pi Z = \pi (({\tau_{\leq {r_0 -1}}} X)[-1])
\simeq (\pi ({\tau_{\leq {r_0 -1}}} X)) [-1] \simeq (\pi X) [-1].$$
Since $Y [1]$ is in ${\mathcal {F}} [1]$ and ${\tau_{\leq {-1}}}:
{\mathcal {F}} \longrightarrow {\mathcal {F}}[1]$ is an equivalence,
there exists an object $N$ in $\mathcal {F}$ such that ${\tau_{\leq
{-1}}} N$ is isomorphic to $Y[1]$. As a consequence, the following
isomorphisms hold in $\mathcal {C}$ $$\pi N \simeq \pi ({\tau_{\leq
{-1}}} N) \simeq \pi (Y[1]) \simeq (\pi Y) [1] \simeq \pi X .$$
Hence we can deduce that for each object $T$ in $\mathcal {C}$,
there exists an object $T'$ in ${\mathcal {F}}$ such that $\pi T'$
is isomorphic to $T$ in $\mathcal {C}$.
\end{proof}

We call $\mathcal {F}$ the {\em fundamental domain}.

\smallskip

$\mathbf{Proof \, of \, the \, main \, theorem \, \ref{1}.}$
\begin{proof}
Proposition \ref{10} and proposition \ref{11} have shown that the
category $\mathcal{C}$ is Hom-finite and $(m+1)$-Calabi-Yau.

Now we only need to show that the object $\pi A$ is an $m$-cluster
tilting object whose endomorphism algebra is isomorphic to the
zeroth homology $H^0 A$ of $A$.

Since $A$ is in the subcategory ${\mathcal{D}}^{{\leq} {0}} \cap \,
^{\perp}{\mathcal{D}}^{{\leq} {-1}}$, $A[i]$ is in
${\mathcal{D}}^{{\leq} {-i}} \cap \, ^{\perp}{\mathcal{D}}^{{\leq}
{-i-1}}$. Thus, the objects $A[i] (1 \leq i \leq m)$ are in the
fundamental domain $\mathcal {F}$. Following proposition \ref{12},
the functor $\pi: {\rm per}A \longrightarrow \mathcal {C}$ induces
an equivalence between $\mathcal {F}$ and $\mathcal {C}$, so we have
that
\begin{flushleft}
$\quad \quad\quad\quad {\rm Hom}_{\mathcal {C}} (\pi A,\pi(A[i]))
\simeq {\rm Hom}_{\mathcal {F}} (A, A[i]) = {\rm Hom}_{\mathcal {D}}
(A, A[i])$
\\
$$ \quad \quad\quad\quad \simeq H^i A = \left\{
                        \begin{array}{ll}
H^0 A, & i = 0; \\
                          \, 0,  & 1 \leq i \leq m .
                        \end{array}
\right.$$
\end{flushleft}
Therefore, the endomorphism algebra of $\pi A$ is isomorphic to the
zeroth homology $H^0 A$ of $A$, and $${\rm Hom}_{\mathcal {C}} (\pi
A,(\pi A[r])) = 0 , \, r = 1, \ldots, m .$$

Let $X$ be an object in $\mathcal {F}$. According to lemma \ref{3},
there exist $m$ triangles where $Q_0, Q_1, \ldots , Q_{m-1}$ are
free $A$-modules and $P_m$ is in $add A$.

Now we will show the following isomorphisms
$${\rm Ext}_{\mathcal
{D}}^1 (P_{m-1}, Y) \simeq {\rm Ext}_{\mathcal {D}}^2 (P_{m-2}, Y)
\simeq \ldots \simeq {\rm Ext}_{\mathcal {D}}^m (X, Y), \, Y \in
{\mathcal {D}}_{\leq 0}. \quad \quad\quad (1)$$ Applying ${\rm
Hom}_{\mathcal {D}} (-, Y[j])$ to the triangle (here we write $P_0$
instead of $X$)
$$P_{m-j+1} {\longrightarrow} Q_{m-j} {\longrightarrow} P_{m-j}
{\longrightarrow} P_{m-j+1} [1] , \, j= 2, \ldots {m},$$ we can get
a long exact sequence $$\ldots \longrightarrow {\rm Hom}_{\mathcal
{D}} (Q_{m-j}[1], Y[j]) \longrightarrow {\rm Hom}_{\mathcal {D}}
(P_{m-j+1}[1], Y[j]) \longrightarrow $$
$$\longrightarrow {\rm Hom}_{\mathcal {D}} (P_{m-j}, Y[j])
\longrightarrow {\rm Hom}_{\mathcal {D}} (Q_{m-j},
Y[j])\longrightarrow \ldots .$$ Since $Q_{m-j}$ are free
$A$-modules, the spaces ${\rm Hom}_{\mathcal {D}} (Q_{m-j}[i],
Y[j])$ are zero for $i = 0, 1$. Therefore, we have the following
isomorphisms
$${\rm Ext}_{\mathcal {D}}^j (P_{m-j}, Y) \simeq {\rm Hom}_{\mathcal {D}}
(P_{m-j}, Y[j]) \simeq {\rm Hom}_{\mathcal {D}} (P_{m-j+1}[1],
Y[j])$$ $$\quad \quad\quad\quad \simeq {\rm Ext}_{\mathcal
{D}}^{j-1} (P_{m-j+1}, Y), \,\quad j=2, \ldots m.$$ It follows that
$(1)$ is true.

Next applying ${\rm Hom}_{\mathcal {D}} (-, Y[j])$ to the triangle
$$P_{m-j} {\longrightarrow} Q_{m-j-1} {\longrightarrow} P_{m-j-1}
{\longrightarrow} P_{m-j} [1] , \quad j= 2, \ldots {m-1} ,$$
similarly we can obtain the following isomorphisms $${\rm
Ext}_{\mathcal {D}}^1 (P_{m-2}, Y) \simeq {\rm Ext}_{\mathcal {D}}^2
(P_{m-3}, Y) \simeq \ldots \simeq {\rm Ext}_{\mathcal {D}}^{m-1} (X,
Y), \quad Y \in {\mathcal {D}}_{\leq 0}.$$

Thus, we can get a list of isomorphisms
$${\rm Ext}_{\mathcal {D}}^1 (P_{m-i}, Y) \simeq {\rm Ext}_{\mathcal
{D}}^{m+1-i} (X, Y), \, 1 \leq i \leq {m-1}, \, Y \in {\mathcal
{D}}_{\leq 0}.$$

Suppose that $Z$ is an object in $\mathcal {C}$ such that the space
${\rm Hom}_{\mathcal {C}} (Z, (\pi A)[i])$ vanishes for each $1 \leq
i \leq m .$ Since the functor $\pi: {\rm per}A \longrightarrow
\mathcal {C}$ induces an equivalence between $\mathcal {F}$ and
$\mathcal {C}$, there exists an object $X$ in $\mathcal {F}$ such
that $\pi X$ is isomorphic to $Z$ in $\mathcal {C}$. Therefore, we
have the following isomorphisms
$${\rm Hom}_{\mathcal {C}} (Z, (\pi A)[i]) \simeq {\rm
Hom}_{\mathcal {C}} (\pi X, (\pi A)[i]) \simeq {\rm
Hom}_{\mathcal{D}} (X, A[i])$$
$$\quad\quad \simeq {\rm Ext}_{\mathcal{D}}^i (X, A), \,\quad 1
\leq i \leq m .$$ Hence, we have $${\rm Ext}_{\mathcal{D}}^1
(P_{m-i}, A) \simeq {\rm Ext}_{\mathcal {D}}^{m+1-i} (X, A) = 0 ,
\quad 1 \leq i \leq {m-1}.$$

As a consequence, the triangle $$P_m {\longrightarrow} Q_{m-1}
{\longrightarrow} P_{m-1} {\longrightarrow} P_m [1]$$ splits, then
the object $P_{m-1}$ is in $add A$. Next the triangle $$P_{m-1}
{\longrightarrow} Q_{m-2} {\longrightarrow} P_{m-2}
{\longrightarrow} P_{m-1} [1]$$ also splits, then the object
$P_{m-2}$ is also in $add A$. By iterated arguments, we can get that
$P_i ~(1 \leq i \leq m)$ and $X$ are all in $add A$. Thus, the
object $Z$, which is isomorphic to $\pi X$ in $\mathcal {C}$, is in
the subcategory $add \pi A$. Hence, the object $\pi A$ is an
$m$-cluster tilting object in the category $\mathcal {C}$.
\end{proof}

\vspace{.3cm}

\section{Generalized higher cluster categories associated to Ginzburg dg categories}

In \cite{Gi06}, V. Ginzburg defined the Ginzburg dg algebra
$\Gamma(Q,W)$ associated to a quiver with potential $(Q,W)$, where
the arrows of the quiver $Q$ are concentrated in degree 0.
Generally, let $Q$ be a graded $k$-quiver such that the set of
objects is finite and $Q(x,y)$ is a finite-dimensional $k$-module
for all objects $x$ and $y$. Let $\mathcal {R}$ be the discrete
$k$-category. Denote by $\mathcal {A}$ the tensor category
$T_{\mathcal {R}}(Q)$. Fixing an integer $n$ and a superpotential
$W$ in the cyclic homology $HC_{n-3}(\mathcal {A})$, {\it i.e.}~a
linear combination of cycles considered up to cyclic permutation
`with signs' of degree $3-n$, the Ginzburg dg category $\Gamma_n
(Q,W)$ is defined in \cite{Ke09} as the tensor category over
$\mathcal {R}$ of the bimodule
$$\widetilde{Q} = Q \oplus Q^{\vee}[n-2] \oplus {\mathcal
{R}}[n-1]$$ endowed with the unique differential which
\begin{itemize}
\item[a)] vanishes on $Q$;
\item[b)] takes the element $a^{\ast}$ of $Q^{\vee}[n-2]$ to the
cyclic derivative $\partial_a W$  for each arrow $a$ in $Q_1$, which
is by definition $\partial _a$ taking a path $p$ to the sum
$\sum_{p=uav} \pm vu$, here the sign is computed by Koszul sign
rule;
\item[c)] takes the element $t_x$ of ${\mathcal {R}}[n-1]$ to
$(-1)^n$ id$_x (\sum_{v \in Q_1} [v,v^{\ast}])$ id$_x$ for each
object $x$ in $Q_0$, where [,] denotes the supercommutator.
\end{itemize}

\begin{thm}[\cite{Ke09}]\label{14}
The Ginzburg dg category $\Gamma_n (Q, W)$ is homologically smooth
and n-Calabi-Yau.
\end{thm}

For simplicity, set $\Gamma^{(n)}$ as the Ginzburg dg category
$\Gamma_n (Q,W)$ associated to a graded quiver with superpotential
$(Q,W)$. Moreover, we assume that the arrows of $Q$ are concentrated
in nonpositive degrees. We denote the minimal degree by $N_Q$.

\begin{thm}\label{20}
Let ${N_Q \geq -m}$ be a nonpositive integer. Suppose that the
zeroth homology of the Ginzburg dg category $\Gamma^{(m+2)}$ is
finite-dimensional. Then the {\em generalized $m$-cluster category
$${\mathcal {C}}^{(m)}_{(Q,W)} = {{\rm
per}{\Gamma^{(m+2)}}/{{\mathcal {D}}^b{\Gamma^{(m+2)}}}}$$
associated to $(Q,W)$} is {\rm Hom}-finite and $(m+1)$-Calabi-Yau.
Moreover, the image of the free module $\Gamma^{(m+2)}$ in
${\mathcal {C}}^{(m)}_{(Q,W)}$ is an $m$-cluster tilting object
whose endomorphism algebra is isomorphic to the zeroth homology of
$\Gamma^{(m+2)}$.
\end{thm}

\begin{proof}
Since the nonpositive integer $N_Q \geq -m$, the elements of $Q^\vee
[m]$ are concentrated in nonpositive degrees. Then the Ginzburg dg
category $\Gamma^{(m+2)}$ has its homology concentrated in
nonpositive degrees. We have that the $p$-th homology $H^p
\Gamma^{(m+2)}$ is zero for each integer $p>0$. By assumption, the
space $H^0 \Gamma^{(m+2)}$ is finite-dimensional. Combining with
theorem \ref{14}, the dg algebra $\Gamma^{(m+2)}$ satisfies the four
properties ($\star$). We apply the main theorem \ref{1} in
particular to $\Gamma^{(m+2)}$. Then the result clearly holds.
\end{proof}

The following corollary considers acyclic quivers with zero
superpotential. In this case, the generalized higher cluster
category ${\mathcal {C}}_{{(Q,0)}}^{(m)}$ recovers the higher
cluster category ${\mathcal {C}}_Q^{(m)}$.

\begin{cor}\label{21}
Let $k$ be an algebraically closed field and $m$ a positive integer.
Suppose that $Q$ is an acyclic ordinary quiver. Then the generalized
$m$-cluster category ${\mathcal {C}}^{(m)}_{(Q, 0)}$ is triangle
equivalent to the orbit category ${\mathcal {C}}_{Q}^{(m)}$ of the
bounded derived category ${\mathcal {D}}$$^b ({\rm {mod}} kQ)$ under
the action of the automorphism ${\tau}^{-1} {\Sigma}^m (= {\nu}^{-1}
{\Sigma}^{m+1})$, where $\Sigma$ (resp. $\nu$) is the suspension
functor (resp. Serre functor) and $\tau$ is the Auslander-Reiten
translation.
\end{cor}

\begin{proof}
Since $Q$ is an acyclic ordinary quiver, the degrees of the arrows
of $\widetilde{Q}$ concentrate in $0, -m, -m-1$ and the homology
$H^{-i} \Gamma^{(m+2)}$ vanishes for each $1 \leq i \leq m-1$.

Since $W$ is zero (in fact, if $m \geq 2$, the only superpotential
is the zero one, otherwise, the degrees of the homogeneous summands
of superpotentials are $1-m (\leq -1)$, while the degrees of the
arrows are zero), the zeroth homology of $\Gamma^{(m+2)}$ is the
finite-dimensional path algebra $kQ$.

Following theorem \ref{20}, the generalized $m$-cluster category
${\mathcal {C}}^{(m)}_{(Q,0)}$ is $(m+1)$-Calabi-Yau, and the image
of $\Gamma^{(m+2)}$ (denoted by $T$) is an $m$-cluster tilting
object whose endomorphism algebra is isomorphic to the
finite-dimensional hereditary algebra $kQ$.

Moreover, from the proof of the main theorem \ref{1}, we know that
the objects ${\Sigma}^i {\Gamma}^{(m+2)} (0 \leq i \leq m)$ are in
the fundamental domain $\mathcal {F}$. Therefore, the following
isomorphisms hold $${\rm Hom}_{\mathcal {C}} (T, {\Sigma}^{-i}T)
\simeq {\rm Hom}_{\mathcal {C}} ({\Sigma}^i T, T) \simeq {\rm
Hom}_{\mathcal {D}} ({\Sigma}^i {\Gamma}^{(m+2)},
{\Gamma}^{(m+2)})$$ $$\quad \quad \quad \quad \quad \quad \quad
\simeq H^{-i} {\Gamma}^{(m+2)} = 0, \quad {\mbox {for each $1 \leq i
\leq m-1$}},$$ where $\mathcal {C}$ denotes the generalized
$m$-cluster category ${\mathcal {C}}^{(m)}_{(Q,0)}$. Hence,
following theorem 4.2 in \cite{KR08}, there is a triangle
equivalence from ${\mathcal {C}}^{(m)}_{(Q,0)}$ to ${\mathcal
{C}}^{(m)}_{Q}$.
\end{proof}

\begin{example}




Suppose $m$ is 2. Let us consider the graded quiver $Q$
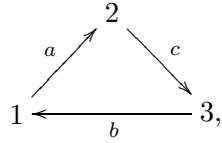
\[
\xymatrix{ & 2 \ar[rd]^c & \\ 1 \ar[ur]^a & & 3 , \ar[ll]^b}
\]
where $deg (a) =-1, deg (b) = deg (c) = 0,$ with superpotential $W =
abc$.

The Ginzburg dg category $\Gamma^{(4)} = \Gamma _4 (Q, W)$ is the
tensor category whose underlying graded quiver is $\widetilde{Q}$
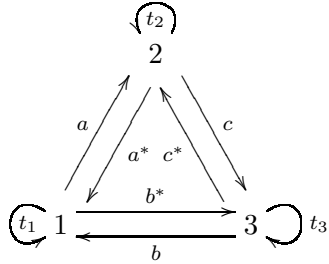
\[
\xymatrix @H=1.2cm { & 2 \ar@(ur,lu)[]^{t_2} \ar@<1ex>[ld]^{a^{\ast}} \ar@<1ex>[rd]^c & \\
1 \ar@(ul,ld)[]^{t_1} \ar@<1ex>[ur]^a \ar@<1ex>[rr]^{b^{\ast}}& & 3
\ar@(ur,rd)[]^{t_3} \ar@<1ex>[ll]^b \ar@<1ex>[ul]^{c^{\ast}}}
\]
where $deg (a^{\ast}) =-1, deg (b^{\ast}) = deg (c^{\ast}) =-2$ and
$deg (t_i) =-3$ for $1 \leq i \leq 3.$ Its differential takes the
following values on the arrows of $\widetilde{Q}$:
\begin{center}
$d(a^{\ast}) = bc, d(b^{\ast}) = ca, d(c^{\ast}) = ab,$ \\ $d(t_1) =
bb^{\ast} + a^{\ast}a, d(t_2) = aa^{\ast} - c^{\ast}c, d(t_3) =
cc^{\ast} - b^{\ast}b.$\end{center}

The zeroth homology $H^0 \Gamma^{(4)}$ equals to the path algebra
with relation $kQ^{(0)}/(bc)$, whose $k$-basis is $\{e_1, e_2, e_3,
 b, c\}$. Therefore, the dimension of $H^0 \Gamma^{(4)}$ is 5.

Following theorem \ref{20}, the image of $\Gamma^{(4)}$ in the
generalized cluster category ${\mathcal {C}}^{(2)}_{(Q,W)}$ is a
2-cluster tilting object, whose endomorphism algebra is given by the
following quiver with the relation
\[
\xymatrix{2 \ar[r]^c & 3 \ar[r]^b & 1\, , & bc = 0.}
\]

In the following, we will show that the generalized cluster category
${\mathcal {C}}^{(2)}_{(Q,W)}$ and the orbit category ${\mathcal
{C}}_{A_3}^{(2)}$ are triangle equivalent.

Let $Q'$ be the quiver \[ \xymatrix{ & 2 \ar[ld]_{\alpha} & \\ 1
\ar[rr]^{\beta} & & 3 , }
\] with $deg(\alpha) = deg(\beta) =0$. We denote the indecomposable
projective mo-\\dule $e_i kQ'$ by $P'_i$, and its corresponding
simple module by $S'_i$ for $1 \leq i \leq 3$. Let $T$ be the almost
tilting module $P'_2 \oplus P'_3$. Its two complements are $P'_1$
and $S'_3$. We write $\overline{T}$ as the direct sum $S'_3 \oplus
P'_2 \oplus P'_3$. Then we have the derived equivalence $${\mathcal
{D}}{\rm End}(\overline{T}) \simeq {\mathcal {D}}({\rm mod} kQ').$$
Following proposition 4.2 in \cite{Ke09}, the derived
4-preprojective dg algebras $\Pi_4 ({\rm End}(\overline{T}),0)$ and
$\Pi_4 (kQ',0)$ are Morita equivalent. Moreover, by theorem 6.3 in
\cite{Ke09}, the derived 4-preprojective dg algebra $\Pi_4 (kQ',0)$
is quasi-isomorphic to the Ginzburg dg category $\Gamma_4 (Q',0)$.

The underlying graded quiver of the algebra ${\rm
End}(\overline{T})$ is \[ \xymatrix{ Q'' & 2 \ar[rd]^{\gamma} &
\\  1 & & 3 , \ar[ll]^{\delta}}
\]
with the relation $\delta \gamma =0$, where $deg(\gamma) =
deg(\delta) =0$. Thus, the algebra ${\rm End}(\overline{T})$ is
quasi-isomorphic to the path algebra of the following graded quiver
$Q'''$ \[ \xymatrix{ & 2 \ar[ld]_{\eta} \ar[rd]^{\gamma} &
\\  1 & & 3 , \ar[ll]^{\delta}}
\] with the differential $d(\eta) =- \delta \gamma$, where $deg(\eta)
=-1$. Following proposition 6.6 in \cite{Ke09}, the derived
4-preprojective dg algebra $\Pi_4 ({\rm End}(\overline{T}),0)$ is
quasi-isomorphic to the tensor category $T_{\mathcal {R}}
(\widetilde{Q'''})$, endowed with the unique differential such that
\begin{center}
$d(\eta) ={\partial}_{{\eta}^{\ast}} W' = \delta \gamma, \,
d({\delta}^{\ast}) ={\partial}_{\delta} W' = \gamma {\eta}^{\ast},
\, d({\gamma}^{\ast}) ={\partial}_{\gamma} W' = {\eta}^{\ast} \delta
, $
\\ $ d(t_1) = \delta
{\delta}^{\ast} + \eta {\eta}^{\ast},\, d(t_2) = {\eta}^{\ast} \eta
- {\gamma}^{\ast} \gamma,\,d(t_3) = \gamma {\gamma}^{\ast} -
{\delta}^{\ast} \delta,$
\end{center}
where $W' = {\eta}^{\ast} \delta \gamma, {\mbox {and}}
~\widetilde{Q'''} = Q''' \oplus (Q''')^{\vee}[2] \oplus {\mathcal
{R}}[3]$.

It is easy to check that the tensor category $T_{\mathcal {R}}
(\widetilde{Q'''})$ endowed with the differential equals the
Ginzburg dg category $\Gamma_4 ({\mathcal {Q}},W')$, where
${\mathcal {Q}}$ is the graded quiver \[ \xymatrix{ & 2
\ar[rd]^{\gamma} &
\\ 1 \ar[ur]^{{\eta}^{\ast}} & & 3 , \ar[ll]^{\delta}}
\]
obtained from $Q'''$ by replacing $\eta$ by ${\eta}^{\ast}$, and
$W'$ is still the superpotential ${\eta}^{\ast} \delta \gamma$.
Obviously, the graded quivers $\mathcal {Q}$ and $Q$ are isomorphic,
while the superpotentials $W'$ and $W$ correspond to each other.
Hence, the derived 4-preprojective dg algebra $\Pi_4 ({\rm
End}(\overline{T}),0)$ is quasi-isomorphic to the Ginzburg dg
category $\Gamma_4 (Q,W)$.

As a consequence, the Ginzburg dg categories $\Gamma_4 (Q,W)$ and
$\Gamma_4 (Q',0)$ are Morita equivalent. Therefore, the generalized
cluster categories ${\mathcal {C}}^{(2)}_{(Q,W)}$ and ${\mathcal
{C}}^{(2)}_{(Q',0)}$ are triangle equivalent. By corollary \ref{21},
we can conclude that the generalized cluster category ${\mathcal
{C}}^{(2)}_{(Q,W)}$ and the orbit category ${\mathcal
{C}}^{(2)}_{A_3}$ are triangle equivalent.
\end{example}

\vspace{.3cm}

\section{Higher cluster categories for algebras \\ of finite global dimension}
Let $A$ be a finite-dimensional $k$-algebra of finite global
dimension. Let $n$ be a positive integer. The bounded derived
category ${\mathcal {D}}$$^b(A)$ admits a right Serre functor
$$\nu_A = - \overset{L}{\otimes}_A DA.$$ Unfortunately, the orbit
category ${\mathcal {O}}_A$ of ${\mathcal {D}}$$^b(A)$ under the
autoequivalence $\nu_A[n]$ is not triangulated in general. Let $X$
be the $A$-$A$-bimodule $DA[-n]$. Let $B$ be the trivial extension
$A \oplus X[-1]$ of $A$ with $A$ in degree 0 and $DA$ in degree
$n+1$. Clearly, the dg $B$-bimodule $DB$ is isomorphic to $B[n+1]$.
The perfect derived category per$B$ is contained in ${\mathcal
{D}}$$^b(B)$ under this construction. It is not hard to check that
for each object $X$ in per$B$ and $Y$ in ${\mathcal {D}}$$^b(B)$,
there is a functorial isomorphism $$D {\rm Hom}_{{\mathcal {D}}B}
(X,Y) \simeq {\rm Hom}_{{\mathcal {D}}B} (Y, X[n+1]).$$ Denote by $p
: B \longrightarrow A$ the canonical projection and $p_{\ast} :
{\mathcal {D}}$$^b(A) \longrightarrow {\mathcal {D}}$$^b(B)$ the
induced triangulated functor. Let $\langle A \rangle_B$ be the thick
subcategory of ${\mathcal {D}}$$^b(B)$ generated by the image of
$p_\ast$. We call the quotient category
$${\mathcal {C}}_A^{(n-1)} = {\langle A \rangle_B}/{{\rm per}B}$$ the
$(n-1)$-{\em cluster category} of $A$. In general, this category
${\mathcal {C}}_A^{(n-1)}$ has infinite-dimensional morphism spaces.
In \cite{Am08} C. Amiot dealt with the case $n \leq 2$. By $\Pi_3 A$
we denote the 3-derived preprojective algebra of $A$ as introduced
in \cite{Ke09}.

\begin{thm}[\cite{Am08}]
Let $A$ be a finite-dimensional $k$-algebra of global dimension
$\leq 2$. If the functor Tor$_2^A (-, DA)$ is nilpotent, then the
cluster category ${\mathcal {C}}_A$ is Hom-finite, 2-Calabi-Yau and
the object $A$ is a cluster tilting object. Moreover, there exists a
triangle equivalence from ${\mathcal {C}}_A$ to the generalized
cluster category ${\mathcal {C}} = {{\rm per}\Pi_3A}/{\mathcal
{D}}$$^b \Pi_3A$ sending the object $A$ onto the image of the
3-derived preprojective algebra $\Pi_3A$ in $\mathcal {C}$.
\end{thm}

In this section, we will investigate the generalization of the above
theorem to the case where $A$ is a finite-dimensional $k$-algebra of
global dimension $\leq n$ (instead of $\leq 2$). Since the
generalization is straightforward, we only list the main steps here,
and leave the proofs to the interested reader.

Let $\mathcal{T}$ be a triangulated category and $\mathcal{N}$ a
triangulated subcategory of $\mathcal{T}$ stable under passage to
direct summands.

\begin{defn}[\cite{Am08}]
Let $X$ and $Y$ be objects in $\mathcal{T}$. A morphism $p: N
\rightarrow X$ is called a {\em local $\mathcal{N}$-cover of $X$
relative to Y} if $N$ is in $\mathcal{N}$ and if it induces an exact
sequence: $$0 \longrightarrow {\mathcal{T}} (X,Y) \longrightarrow
{\mathcal{T}} (N,Y).$$ A morphism $i:X \rightarrow N$ is called a
{\em local $\mathcal{N}$-envelope of $X$ relative to Y} if $N$ is in
$\mathcal{N}$ and if it induces an exact sequence: $$0
\longrightarrow {\mathcal{T}} (Y,X) \longrightarrow {\mathcal{T}}
(Y,N).$$
\end{defn}

\begin{lem}[\cite{Am08}]
Let $X$ and $Y$ be objects of ${\mathcal {D}}$$^b(B)$ such that the
space ${\rm Hom}_{{\mathcal {D}}B} (X,Y)$ is finite-dimensional.
Then there exists a local per$B$-cover of $X$ relative to $Y$.
\end{lem}

Under the assumption of the above lemma, both ${\rm Hom}_{{\mathcal
{D}}B} (N,X)$ and ${\rm Hom}_{{\mathcal {D}}B} (X,N)$ are
finite-dimensional for $N$ in per$B$ and $X$ in ${\mathcal
{D}}$$^b(B)$. Therefore, there exists a local per$B$-envelope of
$X[n+1]$ relative to $Y$. Hence the bilinear form $$\beta'_{X,Y} :
{\rm Hom}_{{\mathcal{C}}_A^{(n-1)}} (X, Y) \times {\rm
Hom}_{{\mathcal{C}}_A^{(n-1)}} (Y, X[n]) \longrightarrow k \quad
{\mbox {for}} X , Y \in {{\mathcal {C}}_A^{(n-1)}}$$ constructed in
the first section of \cite{Am08} is non-degenerate. Therefore, if
${\mathcal {C}}_A^{(n-1)}$ is Hom-finite, then it is $n$-Calabi-Yau
as a triangulated category.

Let us recall the following important properties of the Serre
functor $\nu_A$ of $A$:
\begin{itemize}

\item[$\bullet$] $\nu_A ({\mathcal {D}}_{\geq 0}) \subset {\mathcal {D}}_{\geq
-n}$;

\item[$\bullet$] ${\rm Hom}_{{\mathcal {D}}A} (U,V)$ vanishes for
all $U \in {\mathcal {D}}_{\geq 0}$ and $V \in {\mathcal {D}}_{\leq
{-n-1}}$;

\item[$\bullet$] $\nu_A$ admits an inverse $$\nu_A^{-1} = - \overset{L}{\otimes}_A {\mbox{RHom}}_A
(DA,A),$$ where the homology of the complex RHom$_A (DA,A)$ is given
by
\begin{flushleft}
$\quad \quad\quad\quad H^i {\mbox{RHom}}_A (DA,A) \simeq {\rm
Hom}_{{\mathcal {D}}A} (DA, A[i])$\\ $$ \simeq \left\{
                        \begin{array}{ll}
{\rm
Hom}_{{\mathcal {D}}A} (DA, A), & i = 0, \\
{\rm Ext}^i_A (DA,A), & i = 1, \ldots, n, \\
                          \, 0,  & {\mbox{otherwise}};
                        \end{array}
\right.$$
\end{flushleft}

\item[$\bullet$] $\nu_A^{-1} ({\mathcal {D}}_{\leq 0}) \subset {\mathcal {D}}_{\leq
n}$.

\end{itemize}

Using these properties we obtain the following generalization of
proposition 4.7 of \cite{Am08}.

\begin{prop}\label{18}
Let $A$ be a finite-dimensional $k$-algebra of global dimension
$\leq n$ and $X$ the $A$-$A$-bimodule ${\rm Ext}^n_A (DA,A)$. Then
the endomorphism algebra $\widetilde{A} = {\rm End}_{{\mathcal
{C}}_A^{(n-1)}} (A)$ is isomorphic to the tensor algebra $T_A X$ of
$X$ over $A$. As a consequence, if the category ${\mathcal
{C}}_A^{(n-1)}$ is Hom-finite, then the functor $- {\otimes}_A {\rm
Ext}^n_A (DA,A)$ is nilpotent.
\end{prop}

In fact, the converse statement of the consequence in proposition
\ref{18} is also true. Taking advantage of the above properties of
the Serre functor $\nu_A$, we also have the following variant of
proposition 4.9 of \cite{Am08}.

\begin{prop}
Let $A$ be a finite-dimensional $k$-algebra of global dimension
$\leq n$. The following properties are equivalent:
\begin{itemize}
\item[1)] the category ${\mathcal {C}}_A^{(n-1)}$ is Hom-finite,
\item[2)] the functor $- {\otimes}_A {\rm Ext}^n_A (DA,A)$ is
nilpotent,
\item[3)] the functor Tor$^A_n (-,DA)$ is nilpotent.
\end{itemize}
\end{prop}

Now we give a complete proof for the following well-known lemma.

\begin{lem}\label{22}
Let $A$ be a dg k-algebra. Then for all dg $A$-modules $L, M$, the
objects ${\rm RHom}_A (L,M)$ and ${\rm RHom}_{A^e} (A, {\rm Hom}_k
(L,M))$ are isomorphic in the derived category of dg
$A$-$A$-bimodules.
\end{lem}

\begin{proof}
Let $N$ be an $A$-$A$-bimodule. We construct two maps $\Phi$ and
$\Psi$ as follows
\begin{center}
$\Phi: {\rm Hom}_A (L \otimes_A N,M) \rightarrow {\rm Hom}_{A^e} (N,
{\rm Hom}_k (L,M))$ \\ $\quad \quad \quad \quad \quad \quad \quad
\quad f \longmapsto (\Phi(f) (n) : l \mapsto (-1)^{|l||n|}f(l
\otimes n)),$\\
$\Psi: {\rm Hom}_{A^e} (N, {\rm Hom}_k (L,M)) \rightarrow {\rm Hom}_A (L \otimes_A N,M)$ \\
$\quad \quad \quad \quad \quad \quad \quad \quad \quad\quad g
\longmapsto \Psi(g)(l \otimes n) = (-1)^{|l||n|}g(n)(l).$
\end{center}

\noindent It is not hard to check that $\Phi$ and $\Psi$ are
$A$-$A$-bihomomorphisms homogeneous of degree 0 and satisfy
$$\Phi \Psi = {\mathbf{1}}, \quad \Psi \Phi = {\mathbf{1}}.$$
Moreover, the morphisms $\Phi$ and $\Psi$ commute with the
differentials. Thus, they induce inverse isomorphisms
$${\rm Hom}_{{\mathcal{C}}(A)} (L \otimes_A N,M) \simeq {\rm
Hom}_{{\mathcal{C}}(A^e)} (N, {\rm Hom}_k (L,M)),$$ where
${\mathcal{C}}(B)$ denotes the category of dg $B$-modules for a dg
algebra $B$. The morphisms $\Phi$ and $\Psi$ also induce inverse
isomorphisms
$${\rm Hom}_{{\mathcal{H}}(A)} (L \otimes_A N,M) \simeq {\rm
Hom}_{{\mathcal{H}}(A^e)} (N, {\rm Hom}_k (L,M)),$$ where
${\mathcal{H}}(B)$ denotes the category up to homotopy of dg
$B$-modules for a dg algebra $B$. If we specialize $N$ to $A$, then
we have
$${\rm Hom}_{{\mathcal{H}}(A)} ({\mathbf{p}}L, {\mathbf{i}}M) \simeq {\rm
Hom}_{{\mathcal{H}}(A^e)} (A, {\rm Hom}_k ({\mathbf{p}}L,
{\mathbf{i}}M)),$$ where ${\mathbf{p}}L$ is a cofibrant resolution
of $L$, and ${\mathbf{i}}M$ is a fibrant resolution of $M$.

Now we show that the complex ${\rm Hom}_k ({\mathbf{p}}L,
{\mathbf{i}}M))$ is a fibrant resolution of ${\rm Hom}_k (L,M)$ in
${\mathcal {C}}(A^e)$. Let $\iota: U \rightarrow V$ be a
quasi-isomorphism in ${\mathcal {C}}(A^e)$ which is injective in
each component. We have the isomorphisms
\begin{center}
${\rm Hom}_{{\mathcal {C}}(A^e)} (U, {\rm Hom}_k ({\mathbf{p}}L,
{\mathbf{i}}M)))
\simeq {\rm Hom}_{{\mathcal {C}}(A)} ({\mathbf{p}}L \otimes_A U, {\mathbf{i}}M)$,\\
${\rm Hom}_{{\mathcal {C}}(A^e)} (V, {\rm Hom}_k ({\mathbf{p}}L,
{\mathbf{i}}M))) \simeq {\rm Hom}_{{\mathcal {C}}(A)}
({\mathbf{p}}L \otimes_A V, {\mathbf{i}}M)$.
\end{center}
Since ${\mathbf{p}}L$ is cofibrant, the morphism ${\mathbf{p}}L
\otimes \iota:~ {\mathbf{p}}L \otimes U \rightarrow {\mathbf{p}}L
\otimes V$ is a quasi-isomorphism in ${\mathcal {C}}(A)$ which is
injective in each component. Since ${\mathbf{i}}M$ is fibrant, it
follows that the morphism
$${\rm Hom}_{{\mathcal {C}}(A)} ({\mathbf{p}}L \otimes_A V, {\mathbf{i}}M) \rightarrow
{\rm Hom}_{{\mathcal {C}}(A)} ({\mathbf{p}}L \otimes_A U,
{\mathbf{i}}M)$$ is surjective. Thus, the complex ${\rm Hom}_k
({\mathbf{p}}L, {\mathbf{i}}M)$ is fibrant. Therefore, we have the
following isomorphisms in the derived category of dg
$A$-$A$-bimodules
$${\rm RHom}_A (L,M) \simeq {\rm Hom}_{{\mathcal{H}}(A)} ({\mathbf{p}}L, {\mathbf{i}}M) \simeq {\rm
Hom}_{{\mathcal{H}}(A^e)} (A, {\rm Hom}_k ({\mathbf{p}}L,
{\mathbf{i}}M))$$ $$ \simeq {\rm RHom}_{A^e} (A, {\rm Hom}_k
(L,M)).\quad \quad \,\, \quad$$
\end{proof}

\begin{lem}\label{19}
Assume that $A$ is a proper (i.e. ${\rm dim}_k H^{\ast} A < \infty$)
dg algebra. Then the objects ${\rm RHom}_A (DA,A)$ and ${\rm
RHom}_{A^e} (A,A^e)$ are isomorphic in the derived category of dg
$A$-$A$-bimodules.
\end{lem}

\begin{proof}
If we particularly choose $L$ as $DA$ and $M$ as $A$ in lemma
\ref{22}, then we have the isomorphism in the derived category of dg
$A$-$A$-bimodules
$${\rm RHom}_A (DA,A) \simeq {\rm RHom}_{A^e} (A, {\rm Hom}_k
(DA,A)).$$

Since $A$ is proper, the object $A^e ~(=A^{op} {\otimes}_k A)$ is
quasi-isomorphic to $A^{op} {\otimes}_k D(DA)$ and $DA$ is perfect
over $k$. Therefore, we have the quasi-isomorphisms $$A^e \simeq
A^{op} {\otimes}_k D(DA) \simeq {\rm Hom}_k (DA,A).$$

As a result, we have the following isomorphisms in the derived
category of dg $A$-$A$-bimodules
$${\rm RHom}_{A^e} (A,A^e) \simeq {\rm RHom}_{A^e} (A, {\rm Hom}_k
(DA,A)) \simeq {\rm RHom}_A (DA,A).$$
\end{proof}

Let $\Theta$ be a cofibrant resolution of the dg $A$-bimodule ${\rm
RHom}_A (DA,A)$. Therefore, following \cite{Ke09}, the derived
$(n+1)$-preprojective algebra is defined as $$\Pi_{n+1} (A) = T_A
(\Theta [n]).$$ It is homologically smooth and $(n+1)$-Calabi-Yau as
a bimodule. Moreover, the complex ${\rm RHom}_A (DA,A)[n]$ has its
homology concentrated in nonpositive degrees $-n, \ldots, -1, 0$,
and $$H^0 (\Theta [n]) \simeq H^0 ({\rm RHom}_A (DA,A)[n]) \simeq
H^n ({\rm RHom_A (DA,A)}) \simeq {\rm Ext}^n_A (DA,A).$$ Thus, the
homology of the dg algebra $\Pi_{n+1} (A)$ vanishes in positive
degrees, and we have the following isomorphisms
$$H^0 (\Pi_{n+1} (A)) \simeq T_A (H^0 (\Theta [n])) \simeq T_A ({\rm
Ext}^n_A (DA,A)) \simeq \widetilde{A}.$$

In order for the derived $(n+1)$-preprojective algebra to satisfy
the four properties in section 2, we assume that $H^0 (\Pi_{n+1}
(A))$ is finite-dimensional.

\begin{cor}\label{15}
Let $A$ be a finite-dimensional $k$-algebra of global dimension
$\leq n$. If the functor ${\rm Tor}_n^A (-,DA)$ is nilpotent, then
the generalized $(n-1)$-cluster category $\mathcal {C} = {\rm per}$
${\Pi_{n+1}} (A)/{\mathcal {D}}$$^b \Pi_{n+1} (A)$ is Hom-finite,
$n$-Calabi-Yau and the image of the free dg module $\Pi_{n+1} (A)$
is an $(n-1)$-cluster tilting object in $\mathcal {C}$.
\end{cor}

\begin{proof}
If the functor ${\rm Tor}_n^A (-,DA)$ is nilpotent, by lemma
\ref{19}, the functor $- {\otimes}_A {\rm Ext}^n_A (DA,A)$ is
nilpotent. Thus, the zeroth homology of $\Pi_{n+1} (A)$ is
finite-dimensional. Now we apply theorem \ref{1} in particular to
the derived $(n+1)$-preprojective algebra $\Pi_{n+1} (A)$, then this
corollary holds.
\end{proof}

\begin{thm}\label{16}
Let $A$ be a finite-dimensional $k$-algebra of global dimension
$\leq n$. If the functor ${\rm Tor}_n^A (-,DA)$ is nilpotent, then
the $(n-1)$-cluster category ${\mathcal {C}}_A^{(n-1)}$ of $A$ is
Hom-finite, $n$-Calabi-Yau and the image of $A_B$ is an
$(n-1)$-cluster tilting object in ${\mathcal {C}}_A^{(n-1)}$.
\end{thm}

\begin{proof}
Similarly as \cite{Am08}, we will construct a triangle equivalence
between the $(n-1)$-cluster category ${\mathcal {C}}_A^{(n-1)}$ of
$A$ and the generalized $(n-1)$-cluster category $\mathcal {C}$ of
$\Pi_{n+1} (A)$. Then the statement will follow from corollary
\ref{15}.
\end{proof}

Recall that $\langle A \rangle_B$ denotes the thick triangulated
subcategory generated by $A$ in the derived category ${\mathcal
{D}}$$^b(B)$ of the trivial extension $B$. First we will construct a
triangle equivalence from $\langle A \rangle_B$ to ${\rm per}
\Pi_{n+1} (A)$. Consider the functor ${\rm RHom}_B (A_B, -)$. By a
result in \cite{Ke94}, it induces a triangle equivalence between
$\langle A \rangle_B$ and ${\rm per} C$, where $C$ is the dg algebra
${\rm RHom}_B (A_B, A_B)$. The following lemma is an easy extension
of lemma 4.13 of \cite{Am08}.

\begin{lem}
The dg algebras $\Pi_{n+1} (A)$ and ${\rm RHom}_B (A_B, A_B)$ are
isomorphic objects in the homotopy category of dg algebras.
\end{lem}

As a result, the functor ${\rm RHom}_B (A_B, -)$ induces a triangle
equivalence between $\langle A \rangle_B$ and ${\rm per} \Pi_{n+1}
(A)$, which sends the object $A_B$ to the free module $\Pi_{n+1}
(A)$ and sends the free $B$-module $B$ to the object $A_{\Pi_{n+1}
(A)}$. So the functor also induces an equivalence between the
category ${\rm per}B$ and the thick subcategory $\langle A
\rangle_{\Pi_{n+1} (A)}$ of ${\mathcal {D}} {\Pi_{n+1} (A)}$
generated by $A$. Moreover, as in lemma 4.15 of \cite{Am08}, we
still have that the category $\langle A \rangle_{\Pi_{n+1} (A)}$ is
the bounded derived category ${\mathcal {D}}$$^b {\Pi_{n+1} (A)}$.
Hence, the categories ${\mathcal {C}}_A^{(n-1)}$ and $\mathcal {C}$
are triangle equivalent and theorem \ref{16} holds for arbitrary
positive integers $n$.


\bibliographystyle{amsplain}
\bibliography{stanKeller} 

\end{document}